%% file: manuscriptClaeysWielonsky.tex
\documentclass[12pt,fleqn]{article}
\usepackage{etex, amssymb,latexsym,amsmath,amsfonts,graphicx, ebezier}
\usepackage{pictex,epsfig,times, rotating}
\usepackage{color}

\definecolor{refkey}{rgb}{0,0,1}
\definecolor{labelkey}{rgb}{1,0,0}

   \advance\textheight by \topskip

  \numberwithin{equation}{section}

   \def\Re{{\rm Re \,}}
   
   \def\Ai{{\rm Ai \,}}

   \def\bigO{{\cal O}}

   \def\P2n{{\rm P}_{{\rm II}}^{(n)}}

   \newcommand{\C}{\mathbb{C}}

\newcommand{\N}{\mathbb{N}}

\renewcommand{\O}{\mathcal{O}}

   \newtheorem{theorem}{Theorem}[section]
   \newtheorem{lemma}[theorem]{Lemma}
   \newtheorem{corollary}[theorem]{Corollary}
   \newtheorem{proposition}[theorem]{Proposition}
   
   \newtheorem{Definition}[theorem]{Definition}
   
   \newtheorem{Remark}[theorem]{Remark}
   \newenvironment{remark}{\begin{Remark}\rm}{\end{Remark}}
   \newtheorem{Example}[theorem]{Example}
   
   \newtheorem{Assumptions}[theorem]{Assumptions}

   \newcommand{\e}{\epsilon}

   \newenvironment{proof}%
   {\rm \trivlist \item[\hskip \labelsep{\bf Proof. }]}%
   {\hspace*{\fill}$\Box$\endtrivlist}
   {\rm \trivlist \item[\hskip \labelsep{\bf Proof}]}%
   {\hspace*{\fill}$\Box$\endtrivlist}

\begin{document}
\title{On sequences of rational interpolants of the exponential function with unbounded
interpolation points}
\author{T. Claeys and F. Wielonsky}
\maketitle

\begin{abstract}
We consider sequences of rational interpolants $r_{n}(z)$ of degree
$n$ to the exponential function $e^{z}$ associated to a triangular scheme of
complex points $\{z_{j}^{(2n)}\}_{j=0}^{2n}$, $n>0$, such that, for
all $n$, $|z_{j}^{(2n)}|\leq cn^{1-\alpha}$, $j=0,\ldots,2n$, with
$0<\alpha\leq 1$ and $c>0$. We prove the local uniform convergence
of $r_{n}(z)$ to $e^{z}$ in the complex plane, as $n$ tends to
infinity, and show that the limit distributions of the conveniently
scaled zeros and poles of $r_{n}$ are identical to the corresponding
distributions of the classical Pad\'e approximants. This extends
previous results obtained in the case of bounded (or growing like
$\log n$) interpolation points. To derive our results, we use the
Deift-Zhou steepest descent method for Riemann-Hilbert problems. For
interpolation points of order $n$, satisfying $|z_{j}^{(2n)}|\leq
cn$, $c>0$, the above results are false if $c$ is large, e.g. $c\geq
2\pi$. In this connection, we display numerical experiments showing
how the distributions of zeros and poles of the interpolants may be
modified when considering different configurations of interpolation
points with modulus of order $n$.
\end{abstract}
\\[\baselineskip]
\section{Introduction and main results}
Rational approximants to the exponential function have been the
object of numerous studies in the literature. One motivation comes
from the fact that the approximation of the exponential function
naturally appears in many problems from applied mathematics, like,
for instance, the stability of numerical methods for solving
differential equations, the modeling of time-delay systems to be
found, e.g. in electrical or mechanical engineering, and the
efficient computation of the exponential of a matrix. Another, more
theoretical, motivation comes from the particular properties of the
exponential and its approximants in the framework of function
theory. One classical example of such properties is Pad\'e's theorem
about the convergence of Pad\'e approximants to the exponential, and
its connection with deep results in analytic number theory. Another
typical problem has been the one of finding the rate of rational
approximation to the exponential on the semi-axis, the so-called
$1/9$-conjecture. It attracted the efforts of many authors in the
eighties and was eventually solved in \cite{GR}.

The behavior of Pad\'e approximants to the exponential function has
been studied, among others, in \cite{SV1,SV2,SV3,VC,TEM}, and for
extensions to Hermite-Pad\'e approximants, one may consult
\cite{STA,STA2, KVW,KSVW}. Generalizations to rational interpolants
are investigated in \cite{BOR1, BOR2, BSW, W0, W}. In \cite{W}, it
is shown that rational interpolants to the exponential function with
bounded complex interpolation points (also with points growing like
a logarithm of the degree) converge locally uniformly in the complex
plane, as the degree of the interpolant tends to infinity. The proof
uses the Deift-Zhou steepest descent method for Riemann-Hilbert
problems \cite{DZ, Deift, DKMVZ2, DKMVZ1}. In the present paper, we
consider the case of interpolation points whose modulus may grow
with the degree $n$ of the interpolants, namely like $n^{1-\alpha}$,
$0<\alpha\leq 1$, and we show that the Deift-Zhou method can still
be used to show convergence of the interpolants as $n\to\infty$. For
interpolation points whose growth is linear with respect to the
degree, it is easy to see from the periodicity of the exponential on
the imaginary axis that convergence cannot always hold true. Also,
for the particular case of shifted Pad\'e approximants,
interpolating the exponential at the point $n\xi$, $\xi\in\C $, it
is possible to give a necessary and sufficient condition on $\xi$
for convergence to hold true, see \cite{W1}.

Let us now describe our findings in more detail. Given a triangular
sequence of complex interpolation points
$\{z_j^{(n_{1}+n_{2})}\}_{j=0}^{n_{1}+n_{2}}$, $n_{1}+n_{2}>0$, we
are interested in the behavior, as $n_{1}+n_{2}$ becomes large,  of
the rational function $\frac{p_{n_{1}}}{q_{n_{2}}}$, with
$p_{n_{1}}, q_{n_{2}}$ polynomials satisfying the conditions:
\begin{equation}\label{degree}
\text{(i) }
\deg p_{n_{1}}\leq n_{1},~\deg q_{n_{2}}\leq n_{2},
\end{equation}
\begin{equation}
\label{interpolation} \text{(ii) }
e_{n_{1},n_{2}}(z):=p_{n_{1}}(z)e^{-z/2}+q_{n_{2}}(z)e^{z/2}=\bigO(\omega_{n_
{1}+n_{2}+1}(z)),
\end{equation}
\begin{equation*}
\quad~~\text{as }z\to z_j^{(n_{1}+n_{2})}, ~j=0, \ldots,
n_{1}+n_{2},
\end{equation*}
with
\begin{equation}
\omega_{n_{1}+n_{2}+1}(z)=\prod_{j=0}^{n_{1}+n_{2}}(z-z_j^{(n_{1}+n_{2})}).
\end{equation}
For any choice of (possibly multiple) interpolation points, nontrivial polynomials $p_{n_ {1}}$ and $q_{n_{2}}$, such that
(\ref{degree})--(\ref{interpolation}) hold true, always exist since
these conditions are equivalent to a system of $2n+1$ homogeneous
linear equations with $2n+2$ unknowns.

In this paper, we will only be interested in the diagonal case
\begin{equation}\label{degree-diag}
\deg p_{n}\leq n,\quad \deg q_{n}\leq n,
\end{equation}
and
\begin{equation}\label{interpolation-diag}
p_{n}(z)e^{-z/2}+q_{n}(z)e^{z/2}=\bigO(\omega_{2n+1}(z)), \quad z\to
z_j^{(2n)}, \quad j=0, \ldots, 2n,
\end{equation}
though the general case could be studied similarly. As we will see
in the sequel, even if we restrict ourselves to the diagonal case,
pairs of polynomials of type $(n-1,n+1)$, that is of degrees
respectively less than or equal to $n-1$ and $n+1$, will show up in
the study.

Let us write
\begin{equation}\label{def-rho}
\rho_n:=\max\{|z_j^{(2n)}|: j=0, \ldots, 2n\},\qquad n\in\N.
\end{equation}
If the interpolation points do not grow too rapidly with $n$, i.e.\
if there exist constants $0<\alpha\leq 1$ and $c>0$ such that
\begin{equation}\label{logbound}
\rho_n\leq \frac{1-\alpha}{2}\log n + c, \qquad n\in\N,
\end{equation}
it was proved in \cite[Theorem 2.2]{W} that a pair $(p_{n},q_{n})$,
such that (\ref{degree-diag})-(\ref{interpolation-diag}) hold true
(with $n_{1}=n_{2}=n$), satisfies
\begin{equation*}
p_{n}(z)\to-e^{z/2},\quad q_{n}(z)\to e^{-z/2},\end{equation*}
locally uniformly in $\C$, where $q_{n}$ is normalized so that
$q_{n}(0)=1$. In particular, $p_n(z)/q_n(z)$ converges to $-e^z$
uniformly on compact sets in the complex plane as $n\to\infty$. Our
aim is to weaken the condition (\ref{logbound}) to interpolation
points for which there exists $0<\alpha\leq 1$ and $c>0$
(independent of $n$) such that
\begin{equation}\label{nbound}
\rho_n\leq c n^{1-\alpha}, \qquad n\in\N.
\end{equation}
This is our main result.
\begin{theorem}\label{main thm}
Let $z_j^{(2n)}$, $n>0$, $j=0, \ldots, 2n$, be a family of
interpolation points satisfying (\ref{nbound}) with $0<\alpha\leq 1$
and $c>0$.
Let $p_n$ and $q_n$ be polynomials
satisfying (\ref{degree-diag})-(\ref{interpolation-diag}). Then, the
following three assertions hold true:
\begin{enumerate}
\item[(i)] All the zeros and poles of $r_n=p_n/q_n$ tend to infinity,
  as $n$ becomes large, and, more precisely, no zeros and poles of
  $r_n$ lie in the disk $\{z,|z|\leq\rho_{n}\}$, for $n$ large. In particular, dividing equation
  (\ref{interpolation-diag})
by $q_n$, we get
  $r_n=p_{n}/q_{n}$ as a rational interpolant to $-e^z$ satisfying
\begin{equation*}e^z+r_n(z)=\O(\omega_{2n+1}(z)),\qquad\mbox{ as $z\to z_j^{(2n)}$, $j=0, \ldots, 2n$.}\end{equation*}
\item[(ii)] As $n\to\infty$,
\begin{equation} \label{limpq}
p_n(z)\to -e^{z/2},\quad q_n(z)\to e^{-z/2},\quad r_n(z)\to -e^z,
\end{equation}
locally uniformly in $\C$, where $q_n$ is normalized so that
$q_n(0)=1$.
\item[(iii)] for $n$ large,
\begin{equation} \label{limerr}
e^z+r_n(z) =(-1)^{n}\left(\frac{ec_n}{4n}\right)^{2n+1}w_{2n+1}(z)e^{z-1}
\left(1+\O\left(\frac{1}{n^\alpha}\right)\right),
\end{equation}
locally uniformly in $\C$, where $c_{n}$ is a constant that depends only on the interpolation points $z_j^{(2n)}$ and such that
$$c_{n}=1+\bigO\left(\left(\frac{\rho_{n}}{n}\right)^{2}\right),\quad\text{as }n\to\infty.$$
In particular, if $\rho_{n}=\epsilon(n)\sqrt{n}$ with $\epsilon(n)$ which tends to 0 as $n$ tends to infinity, then (\ref{limerr}) can be rewritten as
\begin{equation} \label{limerr2}
e^z+r_n(z) =(-1)^{n}\left(\frac{e}{4n}\right)^{2n+1}w_{2n+1}(z)e^{z-1}
\left(1+\O(\epsilon^{2}(n))+\O\left(\frac{1}{n^\alpha}\right)\right).
\end{equation}
\end{enumerate}
\end{theorem}
\begin{remark}
For the special case of bounded interpolation points, the error estimate (\ref{limerr2}) agrees with the estimate (2.4) of \cite{W} except for a minus sign that was incorrect there.
\end{remark}
\begin{remark}
For the theorem to be true, an assumption on the growth of
$\rho_{n}$ is mandatory. Indeed, if we allow $\rho_{n}$ to grow
linearly in the degree, that is $\rho_{n}\leq cn$ with $c>0$ some
constant, the theorem can be false. For instance, if $c=2\pi$, it
suffices to consider the constant function $r_{n}(z)=1$ which
interpolates the exponential $e^{z}$ at the points $\{\pm2i\pi
j,~j=0,\ldots,n\}$ and does not converge to it as $n$ tends to
infinity. The particular case of shifted Pad\'e approximants also
shows that the theorem is false for linear growth, even for constant
$c$ smaller than $2\pi$. Indeed, denote by $c_{0}=0.66274...$ the
positive real root of the equation
\begin{equation}
\label{crit-eq} \sqrt{z^{2}+1}+\log\frac{z}{1+\sqrt{z^{2}+1}}=0.
\end{equation}
Then, it follows from results in \cite{W1} that shifted Pad\'e
approximants of degree $n$,  interpolating $e^{z}$ at the point
$nc$, where $c$ is any real number with $|c|\geq c_{0}$, does not
converge to $e^{z}$. Still, we conjecture that Theorem \ref{main
thm} remains true if $\rho_n\leq cn$ with $c<c_{0}$.
\end{remark}

The next theorem describes the limit distributions, as $n\to\infty$,
of the zeros of the scaled polynomials $P_{n}$ and $Q_{n}$ defined by
\begin{equation}\label{scaled}
P_{n}(z)=p_{n}(2nz),\quad Q_{n}(z)=q_{n}(2nz).
\end{equation}
For that, we need to introduce critical
trajectories of the quadratic differential $(z^{2}+1)z^{-2}dz^{2}$,
defined by the condition
\begin{equation}
\label{crit} \Re\int_i^z\frac{(\sqrt{s^2+1})_+}{s}ds=0.
\end{equation}
In (\ref{crit}) we assume that the square root has a branch cut
along the path of integration and behaves like $z$ at infinity. By
$(\sqrt{s^2+1})_+$ we denote the $+$ boundary value of the square
root on that path of integration. An explicit integration of the
differential form in the integral actually shows that condition
(\ref{crit}) can be rewritten in the equivalent form
$\Re\left(\eta(z)\right)=0$, with $\eta(z)$ the expression in the
right-hand side of (\ref{crit-eq}).

From the discussion in \cite{W}, it follows that there are four
critical trajectories, see Figure \ref{traject}.
\begin{figure}
\center
\def\svgwidth{11cm}
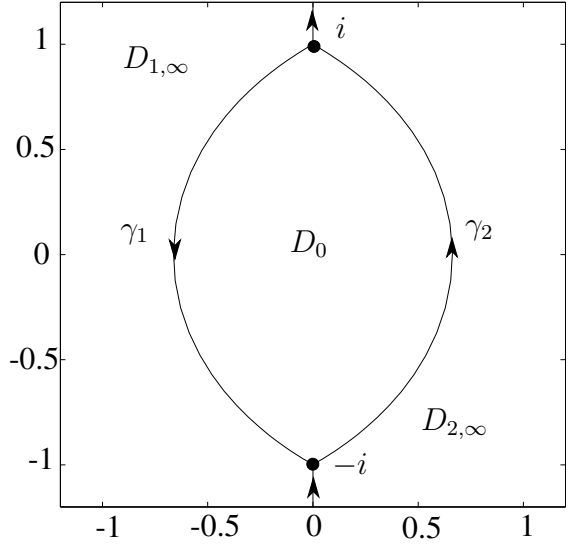
\caption{Critical trajectories satisfying (\ref{crit}) and the
domains $D_{0}$, $D_{1,\infty}$, $D_{2,\infty}$.}
\label{traject}
\end{figure}
We define $\gamma_1$ to be the critical trajectory connecting $i$
with $-i$ in the left half of the complex plane; the other critical
trajectories are the mirror image $\gamma_2$ of $\gamma_1$ with
respect to the imaginary axis, and the vertical half-lines $(\pm
i,\pm i \infty)$. These curves determine three domains that we
denote by $D_{1,\infty}$, $D_{0}$ and $D_{2,\infty}$ as in Figure
\ref{traject}.

Next, we define two positive measures respectively supported on the
curves $\gamma_{1}$ and $\gamma_{2}$, namely
\begin{equation}d\mu_{P}=\frac{1}{i\pi}\frac{(\sqrt{s^{2}+1})_{+}}{s}ds,\quad
d\mu_{Q}=\frac{1}{i\pi}\frac{(\sqrt{(-s)^{2}+1})_{+}}{s}ds,\label{positive
measures}\end{equation}
and, for a polynomial $p$ of degree $n$, we denote by $\nu_{p}$ the
normalized zero counting measure
\begin{equation*}\nu_{p}=\frac1n\sum_{p(z)=0}\delta_{z}.\end{equation*}
Then, the following theorem holds true.
\begin{theorem}
\label{weak-lim} As $n$ tends to infinity, we have
\begin{equation} \label{convnuPQn}
    \nu_{P_n} \stackrel{*}{\to} \mu_P, \qquad
   \nu_{Q_n} \stackrel{*}{\to} \mu_Q,
\end{equation}
where the convergence is in the sense of weak-$^*$ convergence of
measures.
\end{theorem}
The structure of the paper is as follows.
In Section 2, we display a few basic properties of the rational
interpolants and characterize them in terms of the solution of a
specific matrix Riemann-Hilbert (RH) problem. In Section 3, we introduce
different functions that are useful for the steepest descent
analysis of the RH problem, which we perform in Section 4. From this
analysis we derive in Section 5 our main results. Finally, in
Section 6, we present numerical experiments in the case of
interpolation points of order $n$.
\section{Rational interpolants and a Riemann-Hilbert problem}
As said before, polynomials $p_{n}$ and $q_{n}$ such that
(\ref{degree-diag})--(\ref {interpolation-diag}) hold true always
exist. About uniqueness, we have the following simple proposition.
\begin{proposition}\label{irreduc}
The irreducible form of the rational function $p_{n}/q_{n}$,
associated to any pair of polynomials $(p_{n},q_{n})$ satisfying
(\ref{degree-diag})--(\ref{interpolation-diag}), is unique.
\end{proposition}
\begin{proof}
Consider two pairs $(p_{n},q_{n})$ and $(\widetilde p_{n},\widetilde
q_{n})$ satisfying (\ref {degree-diag})--(\ref{interpolation-diag}).
By taking away possible common factors, we get two pairs, each with
coprime polynomials $(p_{n}',q_{n}')$ and $(\widetilde
p_{n}',\widetilde q_{n}')$,
\begin{equation*}\max(\deg p_{n}',\deg q_{n}')= n-d_{1},\quad \max (\deg\widetilde p_{n}',
\deg\widetilde q_{n}'))= n-d_{2}\end{equation*} and sets $J_1, J_2$
such that
\begin{eqnarray*}
p_{n}'(z)+q_{n}'(z)e^{z} & =\bigO\left(\prod_{j\in
J_{1}}(z-z_j^{(2n)})\right),\quad
J_{1}\subset\{0,\ldots,2n\},\quad |J_{1}|\geq 2n+1-d_{1},\\
\widetilde p_{n}'(z)+\widetilde q_{n}'(z)e^{z} &
=\bigO\left(\prod_{j\in J_{2}}(z-z_j^{(2n)}) \right),\quad
J_{2}\subset\{0,\ldots,2n\},\quad |J_{2}|\geq 2n+1-d_{2}.
\end{eqnarray*}
Since $q_{n}'$ (resp. $\widetilde q_{n}'$) does not vanish at
$z_{j}^{(2n)}$, $j\in J_{1}$ (resp. $z_{j}^{(2n)}$, $j\in J_{2}$),
we may divide the above relations by $q_{n}'$ and $ \widetilde
q_{n}'$ respectively, and deduce that
\begin{equation*}p_{n}'(z)\widetilde q_{n}'(z)-\widetilde p_{n}'(z)q_{n}'(z)=\bigO\left(\prod_{j
\in J_{1}\cap J_{2}}(z-z_j^{(2n)})\right),\quad |J_{1}\cap
J_{2}|\geq 2n+1-d_{1}-d_{2}.
\end{equation*}
Since the degree of the polynomial in the left-hand side is at most
$2n-d_{1}-d_{2}$, we may conclude that $p'_{n}/q'_{n}=\widetilde
p'_{n}/\widetilde q'_{n}$, which implies uniqueness of the
irreducible form of the rational function $p_{n}/q_{n}$ as asserted.
\end{proof}
Throughout, we will use the scaled interpolation points
\begin{equation*}\widehat
z_j^{(2n)}:=\frac{z_j^{(2n)}}{2n},\qquad j=0, \ldots,
2n.\end{equation*} Note that, by assumption (\ref{nbound}), we have
\begin{equation*}|\widehat z_{j}^{(2n)}|\leq c\frac{n^{-\alpha}}{2},\qquad j=0, \ldots, 2n.\end{equation*}
Our main object of study will be the following
Riemann-Hilbert problem.
\subsubsection*{RH problem for $Y$}
Find a $2\times 2$ matrix-valued function $Y=Y^{(2n)}:\mathbb
C\setminus\Gamma_n\to \mathbb C^{2\times 2}$, with $ \Gamma_n$ a
counterclockwise oriented closed curve surrounding the scaled
interpolation points $\widehat z_j^{(2n)}$, $j=0, \ldots, 2n,$ such
that
\begin{itemize}
\item[(a)] $Y$ is analytic in $\mathbb C\setminus\Gamma_n$,
\item[(b)] $Y$ has continuous boundary values $Y_+$ ($Y_-$) when
approaching $\Gamma_n$ from the inside (outside) of $\Gamma_n$, and
they are related by the multiplicative jump condition
\begin{equation}\label{RHP Y:b}
Y_+(z)=Y_-(z)\begin{pmatrix}1&e^{-nV_n(z)}\\0&1\end{pmatrix},
\end{equation}
with
\begin{equation}\label{Vn}
V_n(z)=2z+\frac{1}{n}\sum_{j=0}^{2n}\log(z-\widehat z_j^{(2n)}),
\end{equation}
\item[(c)] we have
\begin{equation}\label{RHP Y:c}Y(z)z^{-n\sigma_3}=I+\bigO(z^{-1}),\qquad
\mbox{ as $z\to\infty$},
\end{equation}
where $\sigma_{3}=\begin{pmatrix}1 & 0\\0 & -1\end{pmatrix}$ is the
third Pauli matrix.
\end{itemize}

A motivation for the choice (\ref{Vn}) of the potential $V_{n}$
comes from the fact that
\begin{equation}\label{orthog}
\int_{\Gamma_n}z^{j}P_{n}(z)\frac{e^{-2nz}}
{\Omega_{n}(z)}dz=0,\quad j=0,\ldots,n-1,
\end{equation}
where we set
\begin{equation*}P_{n}(z)=p_{n}(2nz),\qquad\Omega_{n}(z)=\prod_{j=0}^{2n}\left(z-
\widehat z_{j}^{(2n)} \right).\end{equation*} These relations easily
follow from (\ref{degree-diag})-(\ref{interpolation-diag}) and
Cauchy formula. They can be interpreted as orthogonality relations
for the polynomial $P_{n}(z)$ on the contour $ \Gamma_n$ with
respect to the varying weight $e^{-nV_{n}(z)}$.

Next, we prove a proposition concerning existence and uniqueness of
a solution to the RH problem, and relate this solution to the
scaled polynomials $P_n, Q_n$ defined in (\ref{scaled}),
and the scaled remainder term
\begin{equation*}E_{n}(z)=P_{n}(z)e^{-nz}+Q_{n}(z)e^{nz}=\bigO (\Omega_{n}(z)).\end{equation*}
\begin{proposition}\label{prop RHP} The following assertions hold true:\\
(i) There is at most one solution to the RH problem for $Y$. For $n$
large enough, a
solution $Y$ exists.\\
(ii) Let $n\in\mathbb N$ and $z_j^{(2n)}\in\mathbb C$ for $j=0,
\ldots, 2n$. If the RH problem for $Y$ has a solution, and if we
write
\begin{align}
&\label{pn}p_{n}(2nz):=Y_{11}(z), &\mbox{for $z\in\mathbb C\setminus\Gamma_n$},\\
&\label{qn}q_n(2nz):=\Omega_{n}(z)Y_{12}(z),
& \mbox{for $z$ outside of $\Gamma_n$},
\end{align}
then $p_n, q_n$ are polynomials with $\deg p_{n}=n$, $\deg q_{n}\leq
n$, and they satisfy
the interpolation conditions (\ref{interpolation-diag}).\\
(iii) Assume $(p_{n},q_{n})$ is a pair satisfying
(\ref{degree-diag})-(\ref{interpolation-diag}) with
\begin{equation*}\deg p_{n}=n,\quad\deg q_{n}\leq n,\end{equation*}
and $(p_{n-1},q_{n+1})$ is a pair satisfying
(\ref{degree})-(\ref{interpolation}) with
\begin{equation*}\deg p_{n-1}\leq n-1,\quad\deg q_{n+1}= n+1.\end{equation*}
Assume also that the normalizations of $p_{n}$ and $q_{n+1}$ are
chosen so that $P_ {n}(z)=p_{n}(2nz)$ and $q_{n+1}(2nz)$ are monic
polynomials. Then,
\begin{align} \label{eq:Yout}
  Y(z)  & =
         \begin{pmatrix}
         P_{n}(z) & \Omega_n^{-1}(z) Q_{n}(z) \\[10pt]
         p_{n-1}(2nz) & \Omega_n^{-1}(z) q_{n+1}(2nz)
                  \end{pmatrix}, & z\text{ outside  }\Gamma_n,\\[10pt]
\label{eq:Yin}
  Y(z)  & =  \begin{pmatrix}
 P_{n}(z) &  \Omega_n^{-1}(z) e^{-nz}E_{n}(z)\\[10pt]
         p_{n-1}(2nz) &  \Omega_n^{-1}(z)e^{-nz}e_{n-1,n+1}(2nz)
         \end{pmatrix}, & z\text{  inside }\Gamma_n,
\end{align}
solves the RH problem.
\end{proposition}
\begin{proof}
Uniqueness of a solution $Y$ to the RH problem follows easily from
Liouville's theorem, implying first that $\det Y(z)=1$ everywhere in
the complex plane, and second, that the product $\widetilde
YY^{-1}(z)$, where $\widetilde Y$ is another solution, can only
equal $I$, the identity matrix, since $\widetilde YY^{-1}(z)$ has no
jump and behaves like $I+ \bigO (1/z)$ at infinity. The fact that a
solution exists for $n$ large is a consequence of the steepest
descent analysis to be done in Section \ref{descent}.

We now show assertion (ii). The first column of the jump matrix in
(\ref{RHP Y:b}) is $\begin{pmatrix}1\\0\end{pmatrix}$, so $Y_{11}$
and $Y_{21}$ have no jump across $\Gamma_n$, and they are entire
functions. The asymptotic condition (c) of the RH problem tells us
that $Y_{11}$ is a monic polynomial of degree $n$, which we denote
by $P_{n}$, and that $Y_{21}$ is a polynomial of degree at most
$n-1$, which we denote by $\widehat p_{n-1}$. From the jump relation
(\ref{RHP Y:b}), it follows that
\begin{equation*}
Y_{12}(z)=\frac{1}{2\pi i}\int_{\Gamma_n}
\frac{P_{n}(s)e^{-nV_n(s)}}{s-z}ds.
\end{equation*}
Calculating the integrals using residue arguments and the precise
form (\ref{Vn}) of $V_n $, we find, for $z$ outside $\Gamma_n$, that
$Y_{12}$ is of the form
\begin{equation*}
Y_{12}(z)=\Omega_{n}^{-1}(z)Q_{n}(z)
\end{equation*}
where $Q_{n}$ is a polynomial which has to be of degree at most $n$
because of (\ref{RHP Y:c}). Similarly, we find that $Y_{22}$ is of
the form
\begin{equation*}
Y_{22}(z)=\Omega_{n}^{-1}(z)\widehat q_{n+1}(z),
\end{equation*}
with $\widehat q_{n+1}$ a monic polynomial of degree $n+1$. From the
jump relation (\ref{RHP Y:b}), it then follows that
\begin{align}
&\label{interpol1}Y_{12}(z)=e^{-nV_n(z)}\left(P_{n}(z)+e^{2nz}Q_{n}(z)
\right),\\
&\label{interpol2}Y_{22}(z)=e^{-nV_n(z)}\left(\widehat
p_{n-1}(z)+e^{2nz}\widehat q_{n+1} (z)\right),
\end{align}
for $z$ inside $\Gamma_n$. Since $Y_{12}$ is analytic inside
$\Gamma_n$, it follows from (\ref{interpol1}) and the definition
(\ref{Vn}) of $V_{n}$ that the polynomials
$p_{n}(z)=P_n(\frac{z}{2n})$ and $q_{n}(z)=Q_n(\frac{z}{2n})$
satisfy the interpolation conditions (\ref{interpolation-diag}), as
asserted.

Assertion (iii) is easily checked. We leave the details to the
reader.
\end{proof}
\begin{corollary}\label{normal}
For $n$ large enough, polynomials $p_{n}$ and $q_{n}$ satisfying
(\ref{degree-diag})-(\ref {interpolation-diag}) exist, are unique up
to a normalization constant, and the interpolation problem is normal
in the sense that the polynomials $p_{n}$ and $q_{n}$ have full
degrees,
\begin{equation*}\deg p_{n}=n,\qquad\deg q_{n}=n.\end{equation*}
\end{corollary}
\begin{proof}
For $n$ large, we know from assertion (i) of Proposition \ref{prop
RHP} that the RH problem admits a solution $Y$ which is unique. By
assertion (ii) of the same proposition, the polynomials $p_{n}$ and
$q_{n}$ defined by (\ref{pn}) and (\ref{qn}) satisfy the
interpolation conditions
(\ref{degree-diag})-(\ref{interpolation-diag}) with $\deg p_{n}=n$.
Assume there exists another pair $(\widetilde p_{n},\widetilde
q_{n})$, not a scalar multiple of the pair $(p_{n},q_{n})$,
satisfying (\ref{degree-diag})-(\ref{interpolation-diag}). Then, by
Proposition \ref{irreduc}, there exists a polynomial $r$, $\deg
r=d$, $0<d<n$, $p_{n}=r\widetilde p_{n}$, $q_{n}=r\widetilde q_{n}$.
By considering any other polynomial $\widetilde r$ of degree $d$, we
get a pair $\widehat p_{n}=\widetilde r\widetilde p_{n}$, $ \widehat
q_{n}=\widetilde r\widetilde q_{n}$ satisfying
(\ref{degree-diag})-(\ref {interpolation-diag}), with $\deg \widehat
p_{n}=n$, different from the pair $(p_{n},q_{n})$. Moreover, without
loss of generality, we may assume that $\widehat p_{n}(2nz)$ is
monic. Then, the matrix $\widehat Y$ defined by
\begin{align*}
  \widehat Y(z)  & =
         \begin{pmatrix}
         \widehat p_{n}(2nz) & \Omega_n^{-1}(z) \widehat q_{n}(2nz) \\[10pt]
         Y_{21}(z) & Y_{22}(z)
                  \end{pmatrix}, & z\text{ outside  }\Gamma_n,\\[10pt]
  Y(z)  & =  \begin{pmatrix}
 \widehat p_{n}(2nz) &  \Omega_n^{-1}(z) e^{-nz}\widehat E_{n}(z)\\[10pt]
         Y_{21}(z) &  Y_{22}(z)
         \end{pmatrix}, & z\text{  inside }\Gamma_n,
\end{align*}
is different from $Y$ and, using assertion (iii) of Proposition
\ref{prop RHP}, we see that it also solves the RH problem. This is a
contradiction with the uniqueness of a solution $Y$. We may then
conclude that, for $n$ large, there exists, up to a multiplicative
constant, a unique pair $(p_{n},q_{n})$ satisfying
(\ref{degree-diag})-(\ref{interpolation-diag}), and also that $\deg
p_{n}=n$. The fact that, for $n$ large, $\deg q_{n}=n$ is a
consequence of the symmetry of our interpolation problem, namely
that the pair $(p_{n} (z),q_{n}(z))$ solves
(\ref{degree-diag})-(\ref{interpolation-diag}) with respect to the
interpolation points $\{z_{i}^{(2n)}\}$ if and only if the pair
$(q_{n}(-z),p_{n}(-z))$ solves
(\ref{degree-diag})-(\ref{interpolation-diag}) with respect to the
interpolation points $\{-z_{i} ^{(2n)}\}$.
\end{proof}

All the necessary ingredients to prove our convergence results,
namely Theorem \ref {main thm} and Theorem \ref{weak-lim}, are
contained in the RH problem for $Y$.
We will perform a rigorous asymptotic analysis of the RH problem for
$Y$ using the Deift/Zhou nonlinear steepest descent method \cite{DZ,
Deift, DKMVZ2, DKMVZ1}.

\medskip

Using this method, we will obtain existence and precise large $n$
asymptotics for the matrix $Y(z)$ everywhere in the complex plane.
This allows us to obtain uniform asymptotics for the polynomials
$P_{n}$ and $Q_{n}$, from which asymptotics for the original
polynomials $p_n$ and $q_n$ follow. Such asymptotics were obtained
in \cite{W} for interpolation points satisfying (\ref{logbound}),
and can be obtained similarly in our more general situation.

\section{Construction of the $g$-function}\label{Constrg}

Assume that, for each $n>0$, a suitable oriented curve
$\gamma_{1,n}$, with endpoints $a_{n}$ and $b_{n}$, is given. We
want to construct a $g$-function $g_{n}(z)$ satisfying the
conditions
\begin{itemize}
\item[(a)] $e^{g_{n}}:\mathbb C\setminus \gamma_{1,n}\to \mathbb C$ is
analytic,
\item[(b)] there exists $\ell_{n}\in\mathbb C$ such that
\begin{equation}
\label{var eq
1b}g_{n,+}(z)+g_{n,-}(z)-2z-\frac{1}{n}\sum_{j=1}^{2n}\log
(z-\widehat z_j^{(2n)})+2\ell_{n}=0, \qquad\mbox{ for
$z\in\gamma_{1,n}$,}
\end{equation}
where a branch of the logarithm is chosen which is analytic on
$\gamma_{1,n}$,
\item[(c)] $g_{n}(z)= \log z+\bigO(1)$ as $z\to\infty$.
\end{itemize}
This function will play an essential role in the asymptotic analysis
of the RH problem for $Y$: it will enable us to transform the RH
problem for $Y$ to a RH problem for $T$ which is normalized at
infinity (i.e.\ $T(z)\to I$ as $z\to\infty$) and which has jump
matrices which are convenient for asymptotic analysis.

\medskip

One can construct such a function $g_{n}$ with properties (a)-(c)
for any smooth curve $\gamma_{1,n}$. Indeed, condition (\ref{var eq
1b}) is equivalent to
\begin{equation}
g_{n,+}'(z)+g_{n,-}'(z)=2+\frac{1}{n}\sum_{j=1}^{2n}\frac{1}{z-\widehat
z_j^{(2n)}},
\end{equation}
and the function
\begin{equation}
g_{n}'(z)=\frac{1}{R_{n}(z)}\left(1+\frac{1}{2\pi
i}\int_{\gamma_{1,n}}\frac{R_{n,+}(s)}{s-z}\left(2+\frac{1}{n}\sum_{j=1}^{2n}\frac{1}{s-
\widehat z_j^{(2n)}}\right)ds\right),
\end{equation}
with
\begin{equation}\label{R}
R_{n}(z)=\left((z-a_{n})(z-b_{n})\right)^{1/2},\qquad z\in\mathbb
C\setminus\gamma_{1,n},
\end{equation}
satisfies this condition together with the asymptotic condition
$g_{n}'(z)\sim\frac{1}{z}$ if the branch of $R_{n}$ is chosen which
is analytic off $\gamma_{1,n}$ and which behaves like $z$ as
$z\to\infty$. After a straightforward integral calculation, we get
\begin{equation}\label{density}
\psi_{n}(z):=-\frac{1}{2\pi
i}(g_{n,+}'(z)-g_{n,-}'(z))=-\frac{h_{n}(z)}{2\pi
iR_{n,+}(z)},\qquad\mbox{for $z\in\gamma_{1,n}$,}
\end{equation}
with \begin{equation}\label{h}
h_{n}(z)=a_{n}+b_{n}-2z+\frac{1}{n}\sum_{j=1}^{2n}\frac{R_{n}(\widehat
z_j^{(2n)})}{\widehat z_j^{(2n)}-z}.\end{equation} The $g$-function
is then the multi-valued function
\begin{equation}\label{gn}
g_{n}(z)=\int_{\gamma_{1,n}}\log(z-s)\psi_{n}(s)ds,
\end{equation}
where the branch cut of the logarithm follows $\gamma_{1,n}$ along
$(z,b_{n})$ and then it goes further to infinity.

\medskip

For the particular case of Pad\'e approximants where all
interpolation points $z_{j}^{(2n)} $, $j=0,\ldots,2n$, are equal to
$0$, the above formulas become independent of $n$ and reduce to
\begin{equation}
g'(z)=1+\frac1z-\frac{z^{2}-\left(\frac{a+b}{2}\right)z+R(0)}{zR(z)},
\end{equation}
and
\begin{equation}\label{densitypsi0}
\psi(z)=-\frac{1}{2\pi i R_+(z)}\left(a+b-2z-\frac{2R(0)}{z}\right),
\end{equation}
with
\begin{equation*}
R(z)=((z-a)(z-b))^{1/2},
\end{equation*}
and $a$ and $b$ two points independent of $n$. Now, recall the curve
$\gamma_{1}$ that was defined before Theorem \ref{weak-lim}. In the
Pad\'e case, we will choose that particular curve $\gamma_{1}$ in
the definition of the $g$-function. Its endpoints $i$ and $-i$
satisfy $\psi(i)=\psi(-i)=0$, and we have
\begin{equation}\label{density0}
\psi(z)=\frac{R_+(z)}{\pi i z},
\end{equation}
so that
\begin{equation}\label{def g0}
g(z)=\frac{1}{\pi
i}\int_{\gamma_1}\log(z-s)\frac{(\sqrt{s^2+1})_+}{s}ds
\end{equation}
becomes a complex logarithmic potential associated to a real
measure, recall (\ref{crit}). Actually, it is a positive measure by
(\ref{positive measures}).
Next, if we take a suitable curve $\widetilde\gamma_2$ connecting
$+i$ with $-i$, lying to the right of $\gamma_2$, we have the
important inequality
\begin{equation}\label{var ineq}
\Re(g_{+}(z)+g_{-}(z)-2(z+\log z)+2\ell)\leq 0, \qquad\mbox{ for
$z\in\widetilde\gamma_2$,}
\end{equation}
which is strict for $z\in\widetilde\gamma_2\setminus\{\pm i\}$, see
\cite[Lemma 2.9]{W}. We denote by $\Gamma$ the closed contour
\begin{equation}\label{Gamma}
\Gamma=\gamma_{1}\cup\widetilde\gamma_{2},
\end{equation}
oriented counterclockwise. We return to the general case of complex
interpolation points. From (\ref{h}), we obtain
\begin{equation}\label{system1}
h_{n}(a_{n}):=b_{n}-a_{n}+\frac{1}{n}\sum_{j=1}^{2n}\frac{\sqrt{\widehat
z_j^{(2n)}-b_{n}}}{\sqrt{\widehat z_j^{(2n)}-a_{n}}},
\end{equation}
and
\begin{equation}\label{system2}
h_{n}(b_{n})=a_{n}-b_{n}+\frac{1}{n}\sum_{j=1}^{2n}\frac{\sqrt{\widehat
z_j^{(2n)}-a_{n}}}{\sqrt{\widehat z_j^{(2n)}-b_{n}}}.\end{equation}
In an ideal situation, we would choose $a_{n}$ and $b_{n}$ in such a
way that $h_{n}(a_{n})=h_{n}(b_{n})=0$, like in the Pad\'e case.
However, for a general set of interpolation points $z_j^{(2n)}$
bounded by (\ref{nbound}), it is sufficient if $h_{n}(a_{n})$ and
$h_{n}(b_{n})$ are small enough. More precisely, we consider the
following rescaling of the interpolation points,
\begin{equation*}
\widetilde
z_{j}^{(2n)}=\frac{z_{j}^{(2n)}}{2\rho_{n}}=\frac{\widehat
z_{j}^{(2n)}}{t}, \qquad t=\frac{\rho_{n}}{n}\leq\frac{c}{n^{\alpha}},
\end{equation*}
where $\rho_{n}$ denotes the maximum of the moduli of the $z_{j}^{(2n)}$, recall (\ref{def-rho}). Hence the points $\widetilde z_{j}^{(2n)}$ remain bounded with
$n$, and we construct $a_{n} $ and $b_{n}$  of the form
\begin{equation}\label{expansion a b}
a_{n}=i(1+\sum_{j=1}^k\alpha_jt^j),\qquad
b_{n}=-i(1+\sum_{j=1}^k\beta_jt^j)
\end{equation}
in such a way that
\begin{equation}\label{condition psi}
h_{n}(a_{n})=\bigO(t^{k+1}), \qquad
h_{n}(b_{n})=\bigO(t^{k+1}),\qquad\text{as }t\to 0,
\end{equation}
with $k$ a sufficiently large integer. We will see in Section
\ref{Const-param} that $k>\frac{1}{2\alpha}$ is sufficient.
Substituting (\ref{expansion a b}) in
(\ref{system1})-(\ref{system2}), and expanding up to $\bigO(t^{k})$
gives us a system of $2k$ equations: the linear terms in $t$ yield
\begin{equation}\label{coeff-1}
\alpha_1= -\frac{i}{2n}\sum_{j=1}^{2n}\widetilde z_j^{(2n)},\qquad
\beta_1= -\frac{i}{2n}\sum_{j=1}^{2n}\widetilde z_j^{(2n)},
\end{equation}
and in general the $\bigO(t^m)$-term in (\ref{system1}) (resp.\
(\ref{system2})) gives an expression for $\alpha_m$ (resp.\
$\beta_m$) in terms of $\widetilde z_j^{(2n)}$ for $j=1, \ldots, 2n$
and in terms of $\alpha_j, \beta_j$ for $j=1, \ldots, m-1$. In
particular, our system of $2k$ equations in the unknowns $\alpha_1,
\ldots, \alpha_k$ and $\beta_1, \ldots, \beta_k$ has always a unique
solution. We can choose $k$ in (\ref{condition psi}) as large as we
want, but, as said above, $k>\frac{1}{2\alpha}$ will be sufficient
for us.
\begin{proposition}\label{conv-an-bn}
The points $a_{n}$ and $b_{n}$ tend to $i$ and $-i$ respectively, as
$n$ tends to infinity.
\end{proposition}
\begin{proof}
In view of (\ref{expansion a b}), it is sufficient to prove that the
coefficients $\alpha_{l}$ and $\beta_{l}$, $l=1,\ldots,k$, remain
bounded as $n$ tends to infinity. From (\ref {coeff-1}), this is
true when $l=1$.

For indices $l>2$, we prove the assertion by induction. Let
$\widetilde a_{n}=-ia_{n}$ and $\widetilde b_{n}=ib_{n}$, then
\begin{equation*}h_{n}(a_{n})=-iF(\widetilde a_{n},\widetilde b_{n}),\qquad
h_{n}(b_{n})=-iF(-\widetilde b_{n},-\widetilde
a_{n}),\end{equation*} with
\begin{equation*}
F(x,y)=x+y-\frac1n\sum_{j=1}^{2n}\frac{\sqrt{(x+it\widetilde z_j)
(y-it\widetilde z_j)}}{x+it\widetilde z_j},
\end{equation*}
where we have dropped the superscript $(2n)$ in the $\widetilde
z_{j}$'s for simplicity. For $l=2,\ldots,k$, it is easily checked
that the coefficients of $t^{l}$ in the expansions of $F(\widetilde
a_{n},\widetilde b_{n})$  and $F(-\widetilde b_{n},-\widetilde
a_{n})$ can be written as
\begin{equation}\label{coeff-tl-1}
2\alpha_{l}-2c_{0}-\frac{c_{1}}{n}\sum_{j=1}^{2n}\widetilde
z_{j}-\cdots - \frac{c_{l}}{n}\sum_{j=1}^{2n}\widetilde z_{j}^{l},
\end{equation}
\begin{equation}\label{coeff-tl-2}
-2\beta_{l}-2d_{0}-\frac{d_{1}}{n}\sum_{j=1}^{2n}\widetilde
z_{j}-\cdots - \frac{d_{l}}{n}\sum_{j=1}^{2n}\widetilde z_{j}^{l},
\end{equation}
where the $c_{1}$,\ldots,$c_{l}$ and the $d_{1}$,\ldots,$d_{l}$ are
polynomial expressions in the $\alpha_{1}$,\ldots, $\alpha_{l-1}$,
$\beta_{1}$,\ldots,$\beta_{l-1}$. Hence, the vanishing of
(\ref{coeff-tl-1})-(\ref{coeff-tl-2}) and the boundedness of the
$\widetilde z_{j}$, $j=1,\ldots,2n$, show inductively that the
coefficients $\alpha_{l}$ and $\beta_{l}$, $l=1,\ldots,k$, of
$a_{n}$ and $b_{n}$ remain bounded as $n$ tends to infinity.
\end{proof}
\begin{remark}
The question arises whether the expansions (\ref{expansion a b})
are convergent as $k\to\infty$ for small $t$. If this is true, we
have
\begin{equation}h_{n}(\lim_{k\to\infty}a_{n})=h_{n}(\lim_{k\to\infty}b_{n})=0.\end{equation}
This would enable us to relax the condition (\ref{nbound}) to
\begin{equation}
\rho_n\leq c_1n,
\end{equation}
for a sufficiently small constant $c_1>0$.
\end{remark}
We still have to choose two curves $\gamma_{1,n}$ and
$\widetilde\gamma_{2,n}$, each connecting $a_{n}$ with $b_{n}$.
These two curves will build up the closed contour
$\Gamma_{n}=\gamma_ {1,n}\cup\widetilde\gamma_{2,n}$ which appears
in the Riemann-Hilbert problem. We take the contour $\Gamma_{n}$ as
a local deformation around the points $i$ and $-i$ of the contour
$\Gamma$ defined in (\ref{Gamma}). For that, we fix two sufficiently
small disks $U^{(\pm)}$ surrounding $\pm i$. From Proposition
\ref{conv-an-bn}, for $n$ sufficiently large, $a_{n}$ and $b_{n}$
will lie in those disks. We let $\gamma_{1,n}$ and
$\widetilde\gamma_{2,n}$ respectively coincide with $\gamma_{1}$ and
$\widetilde\gamma_{2}$ for $z$ outside $U^{(\pm)}$, and
near $\pm i$ we can extend $\gamma_{1,n}$ and
$\widetilde\gamma_{2,n}$ arbitrarily, as long as they do not
intersect and have no self-intersections.

Finally, define
\begin{equation}\label{phi_n}
\phi_{n}(z)=-2g_{n}(z)+2z+\frac{1}{n}\sum_{j=1}^{2n}\log(z-\widehat
z_j^{(2n)})-2\ell_{n},
\end{equation}
for $z$ outside of $\gamma_{1,n}$.
For $z\in\gamma_{1,n}$, it follows from (\ref{var eq 1b}) that
$\phi_{n,+}(z)=-g_{n,+}(z)+g_{n,-}(z)$, and, together with
(\ref{density}), this implies that
\begin{equation}\label{def phi}
\phi_{n}(z)=-\int_{a_{n}}^z\frac{h_{n}(s)}{R_{n}(s)}ds,
\end{equation}
with $R_{n}$ and $h_{n}$ the functions respectively defined by
(\ref{R}) and (\ref{h}). The path of integration in (\ref{def phi})
is in $\C\setminus(\gamma_{1,n}\cup(\cup\{\widehat z_j^{(2n)}\}))$
and does not wind around any of the points $\widehat z_j^{(2n)}$. The
function $\phi_{n}(z)$ has logarithmic singularities at $ \widehat
z_j^{(2n)}$, $j=1,\ldots,n$, and a branch cut starting at $a_{n}$
which goes along $\gamma_{1,n}$ and then further to infinity. In the
case of Pad\'e interpolants, the function $\phi_{n}$ simplifies to
\begin{equation*}\phi(z)=2\int_{i}^{z}\frac{\sqrt{s^{2}+1}}{s}ds.\end{equation*}
\begin{proposition}\label{prop-phin}
Let $U$ be a given neighborhood of $0$. Then,
we have
\begin{equation*}\phi_{n}\to\phi,\qquad\text{as}\quad n\to\infty,\end{equation*}
locally uniformly in $\C\setminus(\gamma_{1}\cup U)$. Moreover, for
any constant $C<0$, there exists a neighbourhood $U$ of 0 such that
$\Re(\phi_{n})<C$ in $U$ for $n$ large enough.
\end{proposition}
\begin{proof}
From the dominated convergence theorem, the facts that $a_{n}$ tends
to $i$ and the points $\widehat z_j^{(2n)}$ tend to $0$ as
$n\to\infty$ follows that $\phi_{n}$ tends to $ \phi$ point-wise in
$\C\setminus(\gamma_{1}\cup U)$.
By boundedness of the $\phi_{n}$ outside $U$, we derive that the
convergence is locally uniform. The fact that $\Re(\phi_{n})$ has
logarithmic singularities at the $\widehat z_j^ {(2n)}$ which tend
to 0 implies the second assertion.
\end{proof}
Concerning the function $\phi$, the following lemma about the sign
of its real part will be useful in the sequel.
\begin{lemma}\label{lem-sgn-phi}
\emph{(\cite[Lemma 2.9]{W})} Let $D_{0}$, $D_{\infty,1}$ and
$D_{\infty,2}$ be the open domains delimited by the four critical
trajectories of the Pad\'e case, as depicted in Figure
\ref{traject}. Then, the real part of $\phi$ is negative in
$D_{1,\infty}\cup D_{0}$, and it is positive in $D_{2,\infty}$.
\end{lemma}
Note that, by (\ref{crit}), the real part of $\phi$ vanishes on
$\gamma_{1}$, $\gamma_ {2}$ and the vertical half-lines $(\pm i,\pm
i\infty)$.
\section{Steepest descent analysis of the RH problem}\label{descent}
As usual, the steepest descent analysis consists of a number of
transformations.
\subsection{First transformation $Y\mapsto T$}
Define
\begin{equation}\label{def T}
T(z)=e^{n\ell_{n}\sigma_3}Y(z)e^{-ng_{n}(z)\sigma_3}e^{-n\ell_{n}\sigma_3},
\end{equation}
Note that $\Gamma_{n}$, which is fixed outside $U^{(\pm)}$, will
surround the scaled interpolation points $\widehat z_j^{(2n)}$ for
$n$ large by (\ref{nbound}).
\subsubsection*{RH problem for $T$}
\begin{itemize}
\item[(a)] $T:\mathbb C\setminus\Gamma_{n}\to \mathbb C^{2\times 2}$ is
analytic, with $\Gamma_{n}=\gamma_{1,n}\cup\widetilde\gamma_{2,n}$,
\item[(b)] $T$ has continuous boundary values when
approaching $\Gamma_{n}$, and they are related by
\begin{equation*}
T_+(z)=T_-(z)J_T(z),
\end{equation*}
with \begin{equation*}
J_T(z)=\begin{pmatrix}e^{-n(g_{n,+}(z)-g_{n,-}(z))}&
e^{n(g_{n,+}(z)+g_{n,-}(z)-V_n(z)+2\ell_{n})}
\\0&e^{n(g_{n,+}(z)-g_{n,-}(z))}\end{pmatrix}.
\end{equation*}
\item[(c)] $T(z)=I+\bigO(z^{-1})$ as
$z\to\infty$.
\end{itemize}
Making use of the function $\phi_{n}$ defined in (\ref{phi_n}), we
obtain
\begin{equation*}
J_T(z)=\begin{cases}
\begin{pmatrix}e^{n\phi_{n,+}(z)}&e^{W_n(z)}
\\0&e^{n\phi_{n,-}(z)}\end{pmatrix},&\mbox{ for
$z\in\gamma_{1,n}$},\\
\begin{pmatrix}1&e^{W_n(z)}e^{-n\phi_{n}(z)}
\\0&1\end{pmatrix},&\mbox{ for
$z\in\widetilde\gamma_{2,n}$},
\end{cases}
\end{equation*}
with
\begin{equation}\label{Wn}
W_n(z)=
-\log(z-\widehat z_0^{(2n)}).
\end{equation}

\subsection{Opening of the lens $T\mapsto S$}\label{opening}
On $\gamma_{1,n}$, the jump matrix for $T$ can be factorized:
\begin{multline}
\begin{pmatrix}e^{n\phi_{n,+}(z)}&e^{W_n(z)}
\\0&e^{n\phi_{n,-}(z)}\end{pmatrix}=J_1(z)J_2(z)J_3(z)\\=\begin{pmatrix}1&0
\\e^{n\phi_{n,-}(z)}e^{-W_n(z)}&1\end{pmatrix}\begin{pmatrix}0&e^{W_n(z)}
\\-e^{-W_n(z)}&0\end{pmatrix}\begin{pmatrix}1&0
\\e^{n\phi_{n,+}(z)}e^{-W_n(z)}&1\end{pmatrix}.
\end{multline}
The jump matrix $J_1$ can be continued analytically to a region to
the left of $\gamma_{1,n}$, and $J_3$ to a region to the right of
$\gamma_{1,n}$ (simply by replacing $\phi_{n,+}$ and $ \phi_{n,-}$
by $\phi_{n}$). This enables us to split the jump contour
$\gamma_{1,n}$ into three distinct curves $\gamma_{1,n}',
\gamma_{1,n}, \gamma_{1,n}''$ (from left to right) connecting $+i$
with $-i$: we call this the opening of the lens. Define
\begin{equation}\label{def S}
S(z)=\begin{cases} T(z),&\mbox{ for $z$ outside the lens-shaped
region},\\
T(z)J_1(z),&\mbox{ for $z$ in the left part of the lens},\\
T(z)J_3(z)^{-1},&\mbox{ for $z$ in the right part of the lens.}\\
\end{cases}
\end{equation}

\subsubsection*{RH problem for $S$}
\begin{itemize}
\item[(a)] $S:\mathbb C\setminus\Gamma_{n}\to \mathbb C^{2\times 2}$ is
analytic, with
$\Gamma_{n}=\gamma_{1,n}\cup\widetilde\gamma_{2,n}\cup\gamma_{1,n}'\cup\gamma_
{1,n}''$,
\item[(b)] $S$ has continuous boundary values when
approaching $\Gamma_{n}$, and they are related by
\begin{align*}
&S_+(z)=S_-(z)J_T(z),&\mbox{ for $z\in\widetilde\gamma_{2,n}$},\\
&S_+(z)=S_-(z)J_1(z),&\mbox{ for $z\in\gamma_{1,n}'$},\\
&S_+(z)=S_-(z)J_2(z),&\mbox{ for $z\in\gamma_{1,n}$},\\
&S_+(z)=S_-(z)J_3(z),&\mbox{ for $z\in\gamma_{1,n}''$},
\end{align*}
\item[(c)] $S(z)=I+\bigO(z^{-1})$ as
$z\to\infty$.
\end{itemize}
We choose $\gamma_{1,n}', \gamma_{1,n}''$ to lie sufficiently close
to $\gamma_{1,n}$, and in any case in the region where $\phi_{n}$ is
analytic, away from the scaled interpolation points $\widehat
z_j^{(2n)}$. The jump matrices $J_1$, $J_3$, and $J_T$ decay on the
contours $\gamma_{1,n}'$, $\gamma_{1,n}''$, and
$\widetilde\gamma_{2,n}$ respectively as $n\to\infty$, except in two
small but fixed neighborhoods $U^{(\pm)}$ of $\pm i$. This follows
from Proposition \ref{prop-phin} and Lemma \ref{lem-sgn-phi},
together with the fact that $W_n(z)$ is uniformly bounded on the
jump contour. The only jumps that survive the large $n$ limit, are
the jump on $\gamma_{1,n}$ and the jumps in the vicinity of $\pm i$.
In this perspective one is tempted to believe that the leading order
asymptotic behavior of $S$, away from $\pm i$, will be determined by
the solution to a RH problem with jump $J_2$ on the curve
$\gamma_{1,n}$. Nevertheless, some substantial analysis remains to
be done to turn this into a rigorous argument.

\subsection{Outside parametrix}

Consider the following RH problem, on the contour $\gamma_{1,n}$
connecting $a_{n}$ with $b_{n}$.
\subsubsection*{RH problem for $P^{(\infty)}$:}
\begin{itemize}
   \item[(a)] $P^{(\infty)}:\mathbb{C}\setminus \gamma_{1,n}\to\mathbb{C}^{2\times 2}$ is
       analytic,
   \item[(b)] $P^{(\infty)}_+(z)=P^{(\infty)}_-(z)
       \begin{pmatrix}
           0 & e^{W_n(z)} \\
           -e^{-W_n(z)} & 0
       \end{pmatrix}$,\qquad for $z\in\gamma_{1,n}$,
   \item[(c)] $P^{(\infty)}(z)=I+\bigO(z^{-1})$,\qquad as $z\to\infty$.
\end{itemize}
A solution to this problem is given by
\begin{equation}\label{Pinfty}
   P^{(\infty)}(z)=N^{-1}
   \left(\frac{z-b_{n}}{z-a_{n}}\right)^{-\sigma_3 /4}
   ND_{n}(z)^{-\sigma_{3}},
\end{equation}
for $z\in\mathbb C\setminus\gamma_{1,n}$, with
\begin{equation}\label{defN}
           N=\frac{1}{\sqrt 2}
               \begin{pmatrix}
                   1 & 1 \\
                   -1 & 1
               \end{pmatrix} e^{-\frac{1}{4}\pi i\sigma_3},
       \end{equation}
and $D_{n}(z)$ is the Szeg\H{o} function which is analytic and
non-zero in $\C\setminus \gamma_{1,n}$, and satisfies
\begin{equation*}D_{n,+}(x)D_{n,-}(x)=e^{W_{n}(x)},\quad x\in\gamma_{1,n}.\end{equation*}
An explicit expression for $D_{n}(z)$ is given by
\begin{equation}\label{Dn}
D_{n}(z)=\exp\left(\frac{R_{n}(z)}
   {2\pi i}\int_{\gamma_{1,n}}\frac{W_n(s)}{R_{n,+}(s)(s-z)}ds\right),
\end{equation}
with $R_{n}$ defined by (\ref{R}). This outside parametrix
$P^{(\infty)}=P^{(\infty)}(z;n)$ will determine the leading order
asymptotics of $S(z)$ for $z$ in $\mathbb C\setminus U^{(\pm)}$, as
$n\to\infty$. In $U^{(\pm)}$, we need to construct local
parametrices $P^{(\pm)}$ that determine the leading order
asymptotics of $S$ in those disks. This is the goal of the next
section.

\subsection{Local  Airy parametrices}

We will construct local parametrices in the regions $U^{(\pm)}$
surrounding $\pm i$. These parametrices $P=P^{(\pm)}$ should be
analytic in $\overline U^{(\pm)}\setminus
(\gamma_{1,n}'\cup\gamma_{1,n}\cup\gamma_{1,n}''\cup\widetilde\gamma_{2,n})$,
see Figure \ref{contours}, and they should have exactly the same
jumps as $S$ has on $U^{(\pm)} \cap
(\gamma_{1,n}'\cup\gamma_{1,n}\cup\gamma_{1,n}''\cup\widetilde\gamma_{2,n})$.
\begin{figure}
\center
\def\svgwidth{8cm}
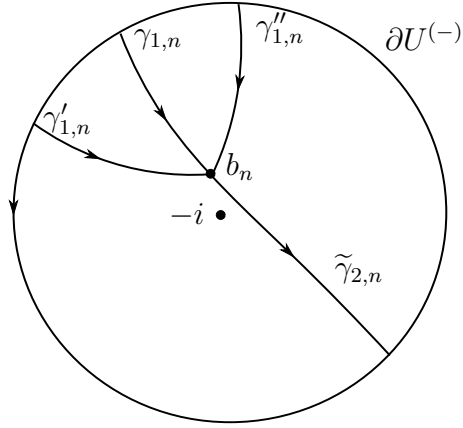
\caption{Contours in the neighborhood $U^{(-)}$ around the point
$-i$.} \label{contours}
\end{figure}
In addition, we aim to construct the parametrices in such a way that
$P^{(\pm)}(z)P^{(\infty)}(z)^{-1}$ is as close as possible to the
identity matrix on $\partial U^{(\pm)}$. As is common for the
construction of local parametrices near the points where the lens
closes, we will build $P$ using the Airy function. If our function
$\phi_{n}(z)$ would behave like $c(z\pm a_{n})^{3/2}$ as $z\to
a_{n}$ and as $c(z\pm b_{n})^{3/2}$ as $z\to b_{n}$, this would be a
standard construction as in \cite{Deift, DKMVZ2, DKMVZ1}.
Unfortunately, this is only the case if
$h_{n}(a_{n})=h_{n}(b_{n})=0$. Therefore we will need some technical
modifications to have suitable parametrices. The reason why we are
still able to construct Airy parametrices, is that
$h_{n}(a_{n})=h_{n}(b_{n})= \bigO(t^{k+1})$ as $n\to \infty$.

\subsubsection{Airy model RH problem}

Define
\[
   y_{j}=y_j(\zeta)=\omega^j\Ai(\omega^j\zeta),\qquad j=0,1,2,
\]
with $\omega=e^{\frac{2\pi i}{3}}$ and $\Ai$ the Airy function. Let
\allowdisplaybreaks{
\begin{align*}
   &A_1(\zeta)=\sqrt{2\pi}e^{-\frac{\pi i}{4}}
       \begin{pmatrix}
           y_0 & -y_2\\
           y_0' & -y_2'
       \end{pmatrix},
       \qquad
   A_2(\zeta)=\sqrt{2\pi}e^{-\frac{\pi i}{4}}\begin{pmatrix}
           -y_1 & -y_2\\
           -y_1' & -y_2'
       \end{pmatrix},
       \\[3ex]
    &A_3(\zeta)=\sqrt{2\pi}e^{-\frac{\pi i}{4}}   \begin{pmatrix}
           -y_2 & y_1\\
           -y_2' & y_1'
       \end{pmatrix},
       \qquad
     A_4(\zeta)= \sqrt{2\pi}e^{-\frac{\pi i}{4}} \begin{pmatrix}
           y_0 & y_1\\
           y_0' & y_1'
       \end{pmatrix},
\end{align*}
} where $y_j'$ denotes the derivative of $y_j$ with respect to
$\zeta$. Since the Airy function is entire, each $A_j$ is an entire
matrix function. Furthermore, using the identity $y_0+y_1+y_2=0$, it
follows that \allowdisplaybreaks{
       \begin{align}
           \label{jumps A: eq1}
           & A_1(\zeta)=A_4(\zeta)
               \begin{pmatrix}
                   1 & 1 \\
                   0 & 1
               \end{pmatrix},
           \\[1ex]
           \label{jumps A: eq2}
           & A_1(\zeta)=A_2(\zeta)
               \begin{pmatrix}
                   1 & 0 \\
                   1 & 1
               \end{pmatrix},
           \\[1ex]
           \label{jumps A: eq3}
           & A_2(\zeta)=A_3(\zeta)
               \begin{pmatrix}
                   0 & 1 \\
                   -1 & 0
               \end{pmatrix},\\
           & A_3(\zeta)=A_4(\zeta)\label{jumps A: eq4}
               \begin{pmatrix}
                   1 & 0 \\
                   1 & 1
               \end{pmatrix}.
       \end{align}
       }
       \begin{figure}
       \center
       \def\svgwidth{8cm}
       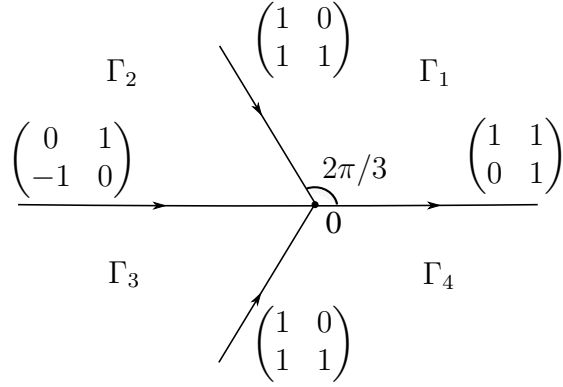
       \caption{The four sectors $(\Gamma_{j})_{j=1,\ldots,4}$, and the contours and
jumps for the matrix $A(\zeta)$.}
       \label{jumps}
       \end{figure}
       From the asymptotics of the Airy function and its derivative,
       \begin{equation}
\label{Airy1} \Ai(\zeta) = \frac{1}{2\sqrt{\pi}} \zeta^{-1/4}
e^{-\frac{2}{3}\zeta^{3/2}}
    \left(1 + \bigO\left(\frac{1}{\zeta^{3/2}}\right)\right),
\end{equation}
\begin{equation}
\label{Airy2} \Ai'(\zeta) = \frac{-1}{2\sqrt{\pi}} \zeta^{1/4}
e^{-\frac{2}{3}\zeta^{3/2}}
    \left(1 + \bigO\left(\frac{1}{\zeta^{3/2}}\right)\right),
\end{equation}
as $\zeta \to \infty$ with $| \arg \zeta | < \pi$, follows that
       \begin{align}\label{RHP:A-c}
           A_j(\zeta) &= \zeta^{-\frac{\sigma_3}{4}}N
               \left[I+\bigO\left(\zeta^{-3/2}\right)\right]
               e^{-\frac{2}{3}\zeta^{3/2}\sigma_3},\quad j=1,\ldots,4,
       \end{align}
as $\zeta\to\infty$ in the sector $\widetilde \Gamma_j$ defined by
       \begin{align}
       &\widetilde \Gamma_1=\{\zeta\in\mathbb C: -\frac{\pi}{3}<\arg\zeta<\pi\},\label{S1}\\
       &\widetilde \Gamma_2=\{\zeta\in\mathbb C: \frac{\pi}{3}<\arg\zeta<\frac{5\pi}{3}\},\\ &\widetilde \Gamma_3=\{\zeta\in\mathbb C: -\frac{5\pi}{3}<\arg\zeta<-\frac{\pi}{3}\},\\
       &\widetilde \Gamma_4=\{\zeta\in\mathbb C:
       -\pi<\arg\zeta<\frac{\pi}{3}\}.\label{S4}
       \end{align}
Let $\Gamma_{j}$, $j=1,\ldots,4$, be the sectors delimited by the
four rays of argument $-\frac{2\pi}{3},0,\frac{2\pi}{3},\pi$, as
shown in Figure \ref{jumps}. Then, it follows from what precedes
that the matrix $A$ such that
\begin{equation*}
A(\zeta)=A_{j}(\zeta),\quad\zeta\in\Gamma_{j},\quad j=1,\ldots,4,
\end{equation*}
admits the jumps shown in Figure \ref{jumps} and has the asymptotic
behavior
\begin{equation}\label{asympt-A}
           A(\zeta) = \zeta^{-\frac{\sigma_3}{4}}N
               \left[I+\bigO\left(\zeta^{-3/2}\right)\right]
               e^{-\frac{2}{3}\zeta^{3/2}\sigma_3},\qquad\text{as }\zeta\to\infty,
\end{equation}
with $N$ the constant matrix defined by (\ref{defN}).
\subsection{Construction of the parametrix in $U^{(-)}$}
\label{Const-param} We search for functions $f_{n}, s_{n}$ in
$U^{(-)}$ such that the function $\phi_{n}$, defined by
(\ref{phi_n}), can be expressed as
\begin{equation}\label{phi f}
\phi_{n}(z)\equiv
\frac{4}{3}f_{n}(z)^{3/2}+2s_{n}(z)f_{n}(z)^{1/2} \mod 2\pi
i,\qquad\mbox{ for $z\in U^{(-)}$}.
\end{equation}
In view of the integral expression (\ref{def phi}) of $\phi_{n}(z)$,
we therefore define $f_ {n}$ by
\begin{equation}
-\frac{4}{3}f_{n}(z)^{3/2}=\int_{b_{n}}^z\frac{h_{n}(s)-h_{n}(b_{n})}{R_{n}(s)}ds,
\end{equation}
and $s_{n}(z)$ by
\begin{equation}
s_{n}(z)=\frac{-1}{2f_{n}(z)^{1/2}}\int_{b_{n}}^z\frac{h_{n}(b_{n})}{R_{n}(s)}ds.
\end{equation}
Then $f_{n}$ and $s_{n}$ are both analytic functions in $U^{(-)}$,
and we have
\begin{align}
&f_{n}(b_{n})=0, &
|f_{n}'(b_{n})|^{3/2}=\frac{|h_{n}'(b_{n})|}{2\sqrt{|b_{n}-a_{n}|}}=\sqrt{2}+\bigO(t)>0,\\
&\label{bounds}s_{n}(z)=\bigO(t^{k+1}),&\mbox{ as $n\to \infty$,
uniformly for $z\in U^{(-)}$}.
\end{align}
For $n$ sufficiently large, $f_{n}+s_{n}$ is a conformal map from
$U^{(-)}$ onto a neighborhood of 0. In the Pad\'e case, we have
\begin{equation*}
\phi(z)=2\int_{-i}^{z}\frac{\sqrt{s^{2}+1}}{s}ds,
\end{equation*}
and it suffices to consider the function $f$ defined by
\begin{equation*}
f(z)=\left[\frac34\phi(z)\right]^{2/3},
\end{equation*}
which is analytic in a neighborhood of $-i$, and where the $2/3$rd
power is taken so that $f(z)$ is real negative for $z\in\gamma_{1}$.
Then, the three contours $\widetilde \gamma_2$, $\gamma_1'$ and
$\gamma_1''$ are chosen so that they are respectively mapped by $f$
on the real positive semi-axis, and on the rays of argument
$-2\pi/3$ and $2\pi/3$. In our situation, and for $n$ large, the map
$f_{n}+s_{n}$ does not send $\gamma_{1,n}$ exactly on the negative
real line, see Figure \ref{maps}. Choosing $\gamma_{1,n}$ as an arc
which tends to $\gamma_{1}$ as $n$ tends to infinity, and since
$f_{n}+s_{n}$ converges uniformly to $f$ in $U^{(-)}$ as $n$ tends
to infinity, we get that $\gamma_{1,n}$ is mapped to an arc
$\lambda_{1,n}$ which tends to the negative real line. Then, the
contours $\gamma_{1,n}'$, $\gamma_{1,n}''$ and
$\widetilde\gamma_{2,n}$ can also be chosen so that they are mapped
by $f_{n}+s_{n}$ to contours $\lambda_{1,n}'$, $ \lambda_{1,n}''$
and $\widetilde\lambda_{2,n}$ tending to the three other rays of the
usual Airy parametrix.
\begin{figure}
\center
\def\svgwidth{18cm}
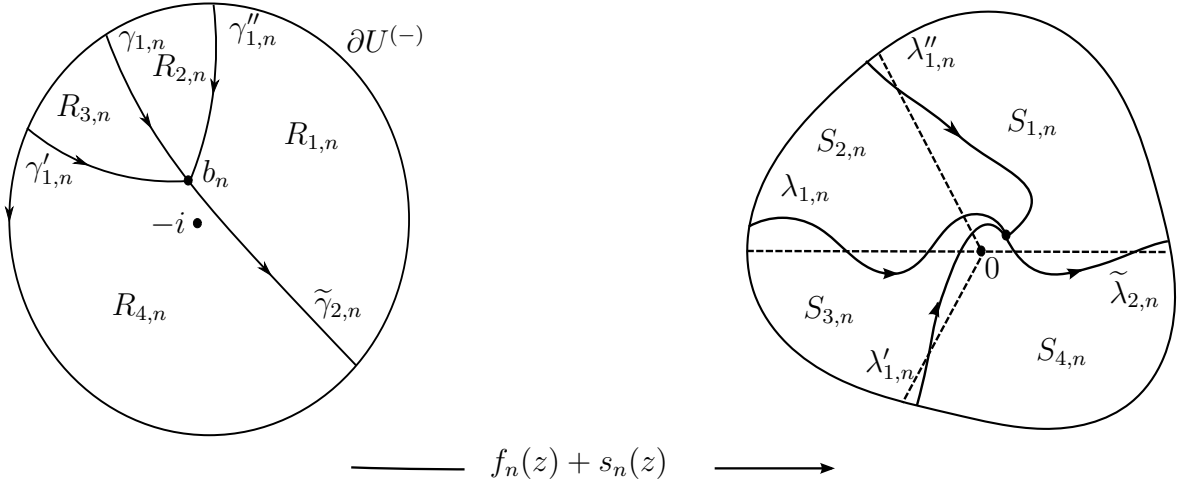
\caption{Images of the domain $U^{(-)}$ and the contours by the
conformal map $f_{n} +s_{n}$. Inside the image domain, contours of
the usual Airy parametrix are drawn in dotted lines.} \label{maps}
\end{figure}

For $E^{(-)}$ some matrix-valued analytic function in $U^{(-)}$ that
we will define in the sequel, let
\begin{equation}\label{def P+}
P^{(-)}(z):=E^{(-)}(z) \widetilde A_{n}(n^{2/3}(f_{n}(z)+s_{n}(z)))
e^{\frac{n}{2}\phi_{n}(z)\sigma_3}e^{-\frac{W_n(z)}{2}\sigma_3},
\end{equation}
with
\begin{equation*}
\widetilde A_{n}(\zeta)=A_{j}(\zeta),\qquad\zeta\in
n^{2/3}S_{j,n},\quad j=1,\ldots,4,
\end{equation*}
and $S_{j,n}$ is the image by the conformal map $f_{n}+s_{n}$ of the
region $R_{j,n}$ in $U^{(-)}$, as indicated in Figure \ref{maps}.
\medskip
As a consequence of (\ref{jumps A: eq1})-(\ref{jumps A: eq4}) and
the right multiplication with the two exponential factors
in (\ref{def P+}), it follows easily that
\begin{align*}
&P_+^{(-)}(z)=P_-^{(-)}(z)J_T(z),&\mbox{ for $z\in\widetilde\gamma_{2,n}$},\\
&P_+^{(-)}(z)=P_-^{(-)}(z)J_1(z),&\mbox{ for $z\in\gamma_{1,n}'$},\\
&P_+^{(-)}(z)=P_-^{(-)}(z)J_2(z),&\mbox{ for $z\in\gamma_{1,n}$},\\
&P_+^{(-)}(z)=P_-^{(-)}(z)J_3(z),&\mbox{ for $z\in\gamma_{1,n}''$}.
\end{align*}

\medskip

Now our main concern is the matching of $P^{(-)}$ with
$P^{(\infty)}$ at $\partial U^{(-)}$. Suppose that $z$ is in region
$R_{j,n}$ and on the boundary of $U^{(-)}$, then one can verify that
$n^{2/3}(f_{n}(z)+s_{n}(z))$ lies in region $\widetilde\Gamma_j$
defined in (\ref{S1})-(\ref{S4}), if $n$ is sufficiently large.
Consequently we can use (\ref{RHP:A-c}) for $z\in\partial U^{(-)}$,
and we obtain
\begin{multline}\label{matching1}
P^{(-)}(z)=E^{(-)}(z)
(n^{2/3}(f_{n}(z)+s_{n}(z)))^{-\frac{\sigma_3}{4}}N\\
\times
               \left[I+\bigO(n^{-1})\right]
               e^{-\frac{2}{3}n(f_{n}(z)+s_{n}(z))^{3/2}\sigma_3}
e^{\frac{n}{2}\phi_{n}(z)\sigma_3}e^{-\frac{W_n(z)}{2}\sigma_3}.
\end{multline}
Substituting (\ref{bounds}), we find by (\ref{phi f}) and
$t\leq cn^{-\alpha}$ that (\ref{matching1}) simplifies to
\begin{equation}\label{matching2}
P^{(-)}(z)=E^{(-)}(z)
n^{-\frac{\sigma_3}{6}}f_{n}(z)^{-\frac{\sigma_3}{4}}N
               \left[I+\bigO(n^{-1})+\bigO(n^{-(k+1)\alpha})+\bigO(n^{1-2(k+1)\alpha})\right]e^{-
\frac{W_n (z)}{2}\sigma_3}.
\end{equation}
For $k>\frac{1}{2\alpha}$, we have
\begin{equation}\label{matching3}
P^{(-)}(z)=E^{(-)}(z)n^{-\frac{\sigma_3}{6}}f_{n}(z)^{-\frac{\sigma_3}{4}}N
               \left[I+\bigO(n^{-2\widehat\alpha})\right]
               e^{-\frac{W_n(z)}{2}\sigma_3},
\end{equation}
with $\widehat\alpha=\min\{\alpha,1/2\}$. Since we want $P^{(-)}$ to
match with the outside parametrix, we define
\begin{equation}\label{def E}
E^{(-)}(z)=P^{(\infty)}(z)e^{\frac{W_n(z)}{2}\sigma_3}N^{-1}f_{n}(z)^{\frac{\sigma_3}{4}}
n^{\frac{\sigma_3}{6}},
\end{equation}
and we need to check that $E^{(-)}$ is analytic in $\overline
U^{(-)}$, since, if not, the jump relations for $P^{(-)}$ would be
violated. By (\ref{Pinfty}), it is easily checked that, indeed,
(\ref{def E}) defines $E^{(-)}$ analytically in $\overline U^{(-)}$.

Since the function $W_n$ defined by (\ref{Wn}) is bounded on
$\partial U^{(\pm)}$, we have
\begin{equation}\label{matching5}
P^{(-)}(z)P^{(\infty)}(z)^{-1}=
               I+\bigO(n^{-2\widehat\alpha}),\qquad \mbox{ for $z\in\partial U^{(-)}, n\to\infty$}.
\end{equation}

\medskip

A similar construction works for the local parametrix $P^{(+)}$ near
$+i$ if we let $P^{(+)}$ be of the form
\begin{equation}\label{def P-}
P^{(+)}(z):=E^{(+)}(z) \sigma_3\widetilde
A_n(-n^{2/3}(f_{n}(z)+s_{n}(z)))\sigma_3
e^{n\frac{\phi_{n}(z)}{2}\sigma_3}e^{-\frac{W_n(z)}{2}\sigma_3},\qquad\mbox{for
$z$ in region $R_{j,n}'$},
\end{equation}  with region $R'_{1,n}$
being the one outside the lens and to the left of it, and the
regions $R'_{2,n}, R'_{3,n}, R'_{4,n}$ occur in order when turning
around $i$ in counterclockwise direction. One can mimic the
construction of $P^{(-)}$ with minor modifications such that
\begin{equation}\label{matching4}
P^{(+)}(z)P^{(\infty)}(z)^{-1}=
               I+\bigO(n^{-2\widehat\alpha}),\qquad \mbox{ for $z\in\partial U^{(+)}$, $n\to\infty$}.
\end{equation}

\subsection{Final transformation}
Now we define
\begin{equation}\label{def R}
R(z)=\begin{cases} S(z)P^{(\infty)}(z)^{-1},&\mbox{ for
$z\in\mathbb C\setminus U^{(\pm)}$,}\\
S(z)P^{(\pm)}(z)^{-1},&\mbox{ for $z\in U^{(\pm)}$.}\end{cases}
\end{equation}
Then $R$ is analytic in $\mathbb C\setminus \Sigma_R$ with
$\Sigma_R$ as shown in Figure \ref{sigmaR}, and it tends to $I$ as
$z\to\infty$.
\begin{figure}
\center
\def\svgwidth{8cm}
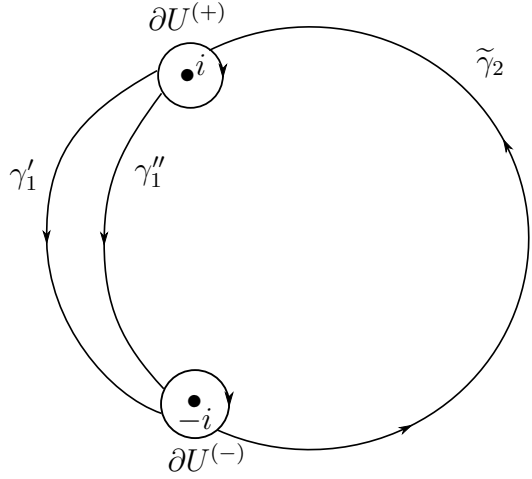
\caption{Contour $\Sigma_{R}$.} \label{sigmaR}
\end{figure}
On $\Sigma_R$, we have
\begin{equation}\label{jump R}
R_+(z)=R_-(z)\left(I+\bigO(n^{-2\widehat\alpha})\right),
\end{equation}
uniformly as $n\to\infty$. The $\bigO(n^{-2\widehat\alpha})$ error term is
present for $z\in\partial U^{(\pm)}$ because of (\ref{matching5})
and (\ref{matching4}). On $\Sigma_R\setminus \partial U^{(\pm)}$ the
jumps are even exponentially small as $n\to\infty$, see the
discussion about the RH problem for $S$ at the end of Section
\ref{opening}. Standard estimates in RH theory show that (\ref{jump
R}) implies existence of the RH solution $R$ for large $n$, and the
asymptotics
\begin{equation}\label{as R}
R(z)=I+\bigO(n^{-2\widehat\alpha}), \qquad\mbox{ as $n\to\infty$,}
\end{equation}
uniformly for $z\in\mathbb C\setminus\Sigma_R$. Reversing the
explicit transformations (\ref{def R}), (\ref{def S}), and (\ref{def
T}), we have existence of $Y$ for $n$ sufficiently large, and we can
also find asymptotics for $Y$ as $n\to\infty$. For $z$ outside the
lens-shaped region and outside $U^{(\pm)}$, we obtain
\begin{eqnarray}
Y_{11}(z)&=&S_{11}(z)e^{ng_{n}(z)}\nonumber\\
&=&\left((1+\bigO(n^{-2\widehat\alpha}))P_{11}^{(\infty)}(z)+\bigO(n^{-2\widehat\alpha})P_{21}^{(\infty)}(z)
\right)e^{ng_{n}(z)}, \label{Y11}
\end{eqnarray}
and
\begin{eqnarray}
Y_{12}(z)&=&S_{12}(z)e^{-ng_{n}(z)}e^{-2n\ell_{n}}\nonumber\\
&=&\left((1+\bigO(n^{-2\widehat\alpha}))P_{12}^{(\infty)}(z)+\bigO(n^{-2\widehat\alpha})P_{22}^{(\infty)}(z)
\right)e^{-ng_{n}(z)}e^{-2n\ell_{n}}, \label{Y12}
\end{eqnarray}

\section{Proof of the main results}
Let
\begin{align*}
r^{(\pm)}(z) &
=\frac12\left(\left(\frac{z+i}{z-i}\right)^{1/4}\pm\left(\frac{z+i}{z-i}\right)^
{-1/4}\right),\\
r_{n}^{(\pm)}(z) &
=\frac12\left(\left(\frac{z-b_{n}}{z-a_{n}}\right)^{1/4}\pm\left(\frac{z-b_{n}}
{z-a_{n}}\right)^{-1/4}\right),
\end{align*}
where the $1/4$-roots are defined outside of $\gamma_{1}$ and
$\gamma_{1,n}$ respectively, and tend to 1 as $z$ tends to infinity.
Note that, since $a_{n}=i(1+\O(n^{-\alpha}))$ and
$b_{n}=-i(1+\O(n^{-\alpha}))$,  as $n\to\infty$, we have
\begin{equation}\label{rn-r}
r_{n}^{(\pm)}(z)=r^{(\pm)}(z)(1+\O(n^{-\alpha})),\quad n\to\infty,
\end{equation}
locally uniformly in $\C\setminus\gamma_{1}$. Also, since $\widehat
z_{0}^{(2n)}=\O(n^{-\alpha})$, we have that
\begin{equation}\label{Dn-D}
D_{n}(z)=D(z)(1+\O(n^{-\alpha})),\quad n\to\infty,
\end{equation}
locally uniformly in $\C\setminus\gamma_{1}$, where
\begin{equation*}D(z)=\exp\left(\frac{R(z)}
   {2\pi i}\int_{\gamma_{1}}\frac{-\log(s)}{R_{+}(s)(s-z)}ds\right).\end{equation*}

We first prove the following proposition which gives asymptotic
estimates for the polynomials $P_{n}(z)$, $Q_{n}(z)$ and the error
function $E_{n}$ in the complex plane, respectively outside of the
curves $\gamma_{1}$, $\gamma_{2}$ and $(\pm i,\pm\infty) $.
\begin{proposition}\label{strong}
As $n\to\infty$, we have
\begin{equation}\label{asympPn}
P_{n}(z)=r^{(+)}(z)D^{-1}(z)e^{ng_{n}(z)}(1+\O(n^{-\alpha})),
\end{equation}
uniformly for $z$ in compact subsets of $\C\setminus\gamma_{1}$,
\begin{equation}\label{asympQn}
Q_{n}(z)=\begin{cases}
-r^{(+)}(z)D^{-1}(z)e^{n(g_{n}(z)-2z)}(1+\O(n^{-\alpha})),\quad &
z\in D_{0}\cup\gamma_{1}\setminus\{i,-i\},\\[10pt]
-i\Omega_{n}(z)r^{(-)}(z)D(z)e^{-n(g_{n}(z)+2l_{n})}(1+\O(n^{-\alpha})),\quad&
z\in\C\setminus\overline{D_{0}},
\end{cases}
\end{equation}
uniformly for $z$ in compact subsets of $\C\setminus\gamma_{2}$.
Furthermore, we have
\begin{equation}\label{asympEn}
E_{n}(z)=\begin{cases}
r^{(+)}(z)D^{-1}(z)e^{n(g_{n}(z)-z)}(1+\O(n^{-\alpha})),\quad &
z\in D_{1,\infty}\cup\gamma_{1}\setminus\{i,-i\},\\[10pt]
-i\Omega_{n}(z)r^{(-)}(z)D(z)e^{n(z-g_{n}(z)-2l_{n})}(1+\O(n^{-\alpha})),\quad&
z\in\C\setminus\overline{D_{1,\infty}},
\end{cases}
\end{equation}
uniformly for $z$ in compact subsets of $\C\setminus(\pm i,\pm
i\infty)$.
\end{proposition}
\begin{proof}
The outside parametrix $P^{(\infty)}(z)$ depends on the endpoints
$a_{n}$ and $b_{n}$. We know that $a_{n}\to i$, $b_{n}\to -i$ as
$n\to\infty$. In view of (\ref{Pinfty}), we can then conclude that
the entries of $P^{(\infty)}(z)$ have modulus uniformly bounded
below and above in compact subsets of $\C\setminus\gamma_{1}$, as
$n\to\infty$. Since $P_{n}(z)=Y_{11} (z)$, we thus get from
(\ref{Y11}) that, locally uniformly,
\begin{equation*}P_{n}(z)=P_{11}^{(\infty)}(z)e^{ng_{n}(z)}(1+\O(n^{-2\widehat\alpha})).\end{equation*}
Moreover, it follows from (\ref{Pinfty}) that
\begin{equation*}P_{11}^{(\infty)}(z)=r_{n}^{(+)}(z)D_{n}^{-1}(z),\end{equation*}
which implies (\ref{asympPn}) because of (\ref{rn-r}) and
(\ref{Dn-D}).

For $Q_{n}$, three different cases need to be considered. First, we
assume $z\in\C \setminus\overline{D_{0}}$. Then we can assume that
the curve $\widetilde\gamma_{2}$ is such that $z$ lies outside the
contour $\Gamma$. From (\ref{eq:Yout}) and (\ref {Y12}), we get
\begin{equation*}Q_{n}(z)=Y_{12}(z)\Omega_{n}(z)=P_{12}^{(\infty)}(z)
\Omega_{n}(z)
e^{-ng_{n}(z)}e^{-2n\ell_{n}}(1+\bigO(n^{-2\widehat\alpha})).\end{equation*}
From the fact that
\begin{equation*}P_{12}^{(\infty)}(z)=-ir_{n}^{(-)}(z)D_{n}(z),\end{equation*}
and using (\ref{rn-r})-(\ref{Dn-D}), we obtain the second estimate
in (\ref{asympQn}). For $z\in D_{0}$, we use the fact that
\begin{equation}Q_{n}(z)=e^{-nz}E_{n}(z)-e^{-2nz}P_{n}(z),\label{QPE}\end{equation}
and we need to find out the dominant term as $n$ gets large. From
(\ref{asympPn}), we have that, as $n$ tends to infinity,
\begin{equation*}\frac1n\log|e^{-2nz}P_{n}(z)|=\Re(g_{n}(z)-2z)+\O(n^{-1}).\end{equation*} Similarly,
from the second formula in (\ref{asympEn}), which we will prove next
and independently, we obtain
\begin{equation*}\frac1n\log|e^{-nz}E_{n}(z)|=\Re(-g_{n}(z)-2l_{n}+\frac1n\log\Omega_{n}
(z))+\O(n^{-1}).\end{equation*} The difference between the
right-hand sides of the two previous estimates equals
$-\Re(\phi_{n}(z))+\O(n^{-1})$ which, in view of Proposition
\ref{prop-phin} and Lemma \ref{lem-sgn-phi}, is positive, locally
uniformly in $D_{0}$, for $n$ large. This implies that the dominant
contribution in (\ref{QPE}) comes from the term $-e^{-2nz}P_n(z)$.
Hence, the first estimate in (\ref{asympQn}) for $z\in D_{0}$
follows from (\ref {asympPn}). It remains to check that the previous
estimate still holds true when $z\in\gamma_{1}\setminus\{-i,i\}$. We
do not give the details here and we simply refer to \cite[Theorem
2.10]{W} where a proof of a similar assertion is given.

For the error function $E_{n}$, and for $z\in D_{0}$, we have
\begin{equation*}E_{n}(z)=e^{nz}\Omega_{n}(z)Y_{12}(z)=P_{12}^{(\infty)}(z)\Omega_{n}
(z)
e^{n(z-g_{n}(z))}e^{-2n\ell_{n}}(1+\bigO(n^{-2\widehat\alpha})),\end{equation*}
which leads to the second estimate in (\ref{asympEn}) when $z\in
D_{0}$. When $z$ is in $D_{1,\infty}$ or $D_{2,\infty}$, we use that
\begin{equation*}E_{n}(z)=P_{n}(z)e^{-nz}+Q_{n}(z)e^{nz}\end{equation*}
and find out which term is dominant in the sum. Using the previous
estimates (\ref {asympPn}) and (\ref{asympQn}), along with
Proposition \ref{prop-phin} and Lemma \ref {lem-sgn-phi}, it turns
out that $P_{n}(z)e^{-nz}$ dominates in $D_{1,\infty}$ while
$Q_{n}(z)e^{nz}$ dominates in $D_{2,\infty}$. Then, (\ref{asympPn})
and (\ref{asympQn}) are used to derive (\ref{asympEn}) in
$D_{1,\infty}$ and $D_{2,\infty}$. Finally, one checks that these
estimates are also valid on the curves $\gamma_{1}$ and $\gamma_
{2}$.
\end{proof}
\subsection{Proof of Theorem \ref{weak-lim}}
\begin{proof}
Applying Rouch\'e theorem on a circle sufficiently large to contain
the curve $\gamma_ {1}$, we obtain, in view of the asymptotic
estimate (\ref{asympPn}), that for $n$ sufficiently large, the
difference between the numbers of poles and zeros of $P_ {n}$
outside of the circle equals the corresponding difference for the
product of functions in the right-hand side of (\ref{asympPn}). This
product has no zero (note that, in view of (\ref{gn}),
$\Re(g_{n})(z)$ is lower bounded if $z$ stays at some distance from
$\gamma_ {1}$) but $e^{ng_n(z)}$ has a pole of order $n$ at
infinity, since $g_{n}(z)=\log z+\O(1)$ as $z\to\infty$. Since
$P_{n}$ has a pole of multiplicity $n$ at infinity, we may thus
conclude that $P_{n}$ has no zero outside of the circle for $n$
large enough. A similar argument using Rouch\'e's theorem shows
that, for any compact in $\mathbb C\setminus \gamma_1$, there exist
$n_0$ such that $P_n$ has no zeros in that compact for $n\geq n_0$.

Now, let us consider the sequence of counting measures
$\nu_{P_{n}}$, $n>0$. We already know that, for $n$ large enough,
all these measures are supported inside a fixed compact set. From
Helly's selection theorem, we may thus select a subsequence of
$\nu_{P_{n}}$ converging in weak-* sense to a measure $\nu$.
Besides, from (\ref{asympPn}), it follows that, as $n\to\infty$,
\begin{equation*}\frac1n\log|P_{n}(z)|=\Re(g_{n}(z))+\O(n^{-1}),
\quad z\in\C\setminus\gamma_{1},\end{equation*} and from the
dominated convergence theorem, we see that the sequence of functions
$g_{n}(z)$ tends to $g(z)$, defined in (\ref{def g0}), point-wise in
$\C\setminus\gamma_ {1}$. Hence,  as $n\to\infty$,
\begin{equation*}\frac1n\log|P_{n}(z)|\to\Re(g(z)),
\quad z\in\C\setminus\gamma_{1},\end{equation*} or equivalently,
\begin{equation*}\int\log|z-s|d\nu_{P_{n}}(s)\to\Re(g(z)),
\quad z\in\C\setminus\gamma_{1}.\end{equation*} Since $\nu_{P_{n}}$
tends to $\nu$ and $P_{n}$ has no zero inside any compact set of
$\C\setminus\gamma_{1}$ for $n$ large enough, the above integral
also tends to $\int\log|z-s|d\nu(s)$, point-wise in $\C\setminus
\gamma_{1}$, as $n\to\infty$, so that
\begin{equation*}\int\log|z-s|d\nu(s)=\int\log|z-s|d\mu_{P}(s),\quad z\in\C\setminus
\gamma_{1},\end{equation*} where we recall that the function $g$ is
the complex logarithmic potential associated to the measure
$d\mu_{P}$. Since $\gamma_{1}$ has two-dimensional Lebesgue measure
0, the unicity theorem \cite[Theorem II.2.1]{SaTo} applies, showing
that $\nu$ and $ \mu_{P}$ are equal. Since $\mu_{P}$ is the only
possible limit of a weakly convergent subsequence, the full sequence
$\nu_{P_{n}}$ converges weakly to $\mu_{P}$.

The similar result for the sequence of counting measures
$\nu_{Q_{n}}$, $n>0$, follows from the symmetry of our interpolation
problem, already mentioned at the end of the proof of Corollary
\ref{normal}.
\end{proof}
\subsection{Proof of Theorem \ref{main thm}}
Before we start the proof, we need a few preliminary results.
\begin{lemma}
Let $z,u\in\C\setminus\gamma_{n,1}$. Then, we have
\begin{equation}
\frac{R_{n}(z)}{2\pi
i}\int_{\gamma_{n,1}}\frac{\log(u-s)}{R_{n,+}(s)(s-z)}ds=
\frac12\log\left(\frac{(w_{1}(u)-w_{2}(z))(w_{1}(z)-w_{2}(u))}{2w_{2}(z)+a_{n}+b_{n}}\right),
\end{equation}
\begin{equation}
\frac{1}{2\pi
i}\int_{\gamma_{n,1}}\frac{(a_{n}+b_{n}-2s)\log(u-s)}{R_{n,+}(s)}ds=
\frac12\frac{(a_{n}-b_{n})^{2}}{2w_{2}(z)+a_{n}+b_{n}},
\end{equation}
where
\begin{equation}\label{def-w}
w_{1}(s)=-s+R_{n}(s),\quad w_{2}(s)=-s-R_{n}(s),\quad s\in\C\setminus
\gamma_{1,n}.
\end{equation}
\end{lemma}
\begin{proof}
We will not give the details for the proof of these formulas. We
just mention that the computations can be performed by using the
change of variables $s\to w$, where $s$ and $w$ satisfy the
algebraic equation
\begin{equation*}w^{2}+2sw+(a_{n}+b_{n})s-a_{n}b_{n}=(w-w_{1}(s))(w-w_{2}(s))=0,\end{equation*}
and then applying the Cauchy formula in the $w$-plane.
\end{proof}
\begin{proposition}\label{prop gD}
The functions $g_{n}$, $D_{n}^{2}$ and the constant $2\ell_{n}$ admit the following explicit
expressions,
\begin{align}\label{expl-g}
g_{n}(z) & =-\frac12\frac{(a_{n}-b_{n})^{2}}{2w_{2}(z)+a_{n}+b_{n}}
+\frac{1}{2n}\sum_{j=1}^{2n}\log\left( \frac{(w_{2}(\widehat
z_{j}^{(2n)})-w_{1}(z)) (w_{1}(\widehat
z_{j}^{(2n)})-w_{2}(z))}{2w_{2}(\widehat
z_{j}^{(2n)})+a_{n}+b_{n}}\right),
\\\label{expl-D}
D_{n}^{2}(z) & =\frac{-2w_{2}(z)-a_{n}-b_{n}}{(w_{2}(\widehat z_{0}^{(2n)})-
w_{1}(z)) (w_{1}(\widehat
z_{0}^{(2n)})-w_{2}(z))},
\\\label{expl-l}
2\ell_{n} & =a_{n}+b_{n}-\frac{1}{2n}\sum_{j=1}^{2n}\log\left(
\frac{2w_{1}(\widehat z_{j}^{(2n)})+a_{n}+b_{n}}
{2w_{2}(\widehat z_{j}^{(2n)})+a_{n}+b_{n}}\right).
\end{align}
\end{proposition}
\begin{proof}
The proofs of (\ref{expl-g}) and (\ref{expl-D}) simply follow from the previous lemma together with the
expression (\ref {density})-(\ref{gn}) of $g_{n}$ and the expression
(\ref{Dn}) of $D_{n}$. The proof of (\ref{expl-l}) follows by plugging (\ref{expl-g}) into (\ref{var eq 1b}) and performing a few calculations that we do not detail.
\end{proof}
{\bf Proof of Theorem \ref{main thm}.} Assertion (i) is a
consequence of the strong asymptotics obtained in Proposition \ref
{strong}. In order to prove (\ref{limpq}) we need to evaluate the
asymptotic estimates (\ref {asympPn}) and (\ref{asympQn}) with $z$
replaced by $z/2n$ where $z$ is a fixed complex number. We get
\begin{align}\label{Pn-norm}
P_{n}\left(\frac{z}{2n}\right) &
=r^{(+)}(0)D^{-1}(0)e^{ng_{n}\left(\frac{z}{2n}\right)}
(1+\O(n^{-\alpha})),\\[10pt]\label{Qn-norm}
Q_{n}\left(\frac{z}{2n}\right) &
=-r^{(+)}(0)D^{-1}(0)e^{ng_{n}\left(\frac{z}{2n}\right)}e^{-z}
(1+\O(n^{-\alpha})).
\end{align}
Since the normalization chosen in Theorem \ref{main thm} is such
that $q_{n}(0)=1$, we can deduce from
(\ref{Pn-norm})-(\ref{Qn-norm}) that
\begin{align}\label{pn-norm}
p_{n}(z) &
=-e^{n\left(g_{n}\left(\frac{z}{2n}\right)-g_{n}(0)\right)}
(1+\O(n^{-\alpha})),\\[10pt]\label{qn-norm}
q_{n}(z) &
=e^{n\left(g_{n}\left(\frac{z}{2n}\right)-g_{n}(0)\right)}e^{-z}
(1+\O(n^{-\alpha})).
\end{align}
As $g_{n}$ tends to $g$ uniformly in a neighbourhood of 0, we have
that
\begin{equation}\label{dev-ngn}
ng_{n}\left(\frac{z}{2n}\right)=ng_{n}(0)+g_{n}'(0)\frac{z}{2}+\O\left(\frac1n\right),
\end{equation}
where the $\O(1/n)$ term is uniform in $z$. Finally, $g_{n}'(0)$
tends to
\begin{equation*}
g'(0)=-\frac{1}{i\pi}\int_{\gamma_{1}}\frac{(\sqrt{s^{2}+1})_{+}}{s^{2}}ds,
\end{equation*}
which, by Cauchy theorem, is easily computed to be 1. Hence, by
plugging (\ref{dev-ngn}) into (\ref{pn-norm})-(\ref{qn-norm}), the
limits (\ref{limpq}) follows, and the convergence is locally uniform
in $\C$.

For the third assertion, note that
\[e^z+r_n(z)=e^z+\frac{P_n(\frac{z}{2n})}{Q_n(\frac{z}{2n})}=e^{\frac{z}{2}}\frac{E_n(\frac{z}{2n})}{Q_n(\frac{z}{2n})}.\]
Substituting (\ref{asympQn}) and (\ref{asympEn}), we obtain by (\ref{rn-r}) and (\ref{Dn-D})
\begin{equation}
e^z+r_n(z)=i\frac{r^{(-)}(0)}{r^{(+)}(0)}\frac{w_{2n+1}(z)}{(2n)^{2n+1}}D(0)^2e^{2z}e^{-2ng_n(\frac{z}{2n})}e^{-2n\ell_n}\left(1+\bigO(n^{-\alpha})\right).
\end{equation}
Evaluating $r^{(\pm)}(0)$ and $D(0)$ using Proposition \ref{prop gD}, and expanding $ng_n(\frac{z}{2n})$, we find
\begin{equation}\label{estim-err}
e^z+r_n(z) =\frac{1}{2}\frac{w_{2n+1}(z)}{(2n)^{2n+1}}e^{z}e^{-2ng_n(0)}
e^{-2n\ell_n}\left(1+\bigO(n^{-\alpha})\right)
\end{equation}
where we have also used that $g'(0)=1$.
From (\ref{expl-g}) and (\ref{expl-l}) we can derive an explicit expression for $2g_n(0)+2\ell_n$, namely,
\begin{equation}\label{expl-c}
2g_n(0)+2\ell_n=-2\sqrt{a_{n}b_{n}}+
\frac{1}{2n}\sum_{j=1}^{2n}\log\left( \frac
{(\widehat z_{j}^{(2n)})^{2}(w_{2}(\widehat z_{j}^{(2n)})-\sqrt{a_{n}b_{n}})
(w_{1}(\widehat z_{j}^{(2n)})+\sqrt{a_{n}b_{n}})}
{(w_{2}(\widehat z_{j}^{(2n)})+\sqrt{a_{n}b_{n}})
(w_{1}(\widehat z_{j}^{(2n)})-\sqrt{a_{n}b_{n}})}\right).
\end{equation}
Moreover, from (\ref{expansion a b}), (\ref{coeff-1}) and the fact that the coefficients $\alpha_{j}$ and $\beta_{j}$ are bounded with respect to $n$ follows that
$$
a_{n}b_{n}=1-\frac{i}{n}\sum_{k=1}^{2n}\widetilde z_k^{(2n)}t+\bigO(t^{2}).
$$
Also, recalling the definitions (\ref{def-w}) of $w_{1}$ and $w_{2}$, we deduce that
\begin{align*}
w_{2}(\widehat z_{j}^{(2n)})-\sqrt{a_{n}b_{n}} & =
-2+\left(\frac{i}{n}\sum_{k=1}^{2n}\widetilde z_k^{(2n)}-\widetilde z_j^{(2n)}\right)t+\bigO(t^{2}),\\
w_{1}(\widehat z_{j}^{(2n)})+\sqrt{a_{n}b_{n}} & =
2-\left(\frac{i}{n}\sum_{k=1}^{2n}\widetilde z_k^{(2n)}+\widetilde z_j^{(2n)}\right)t+\bigO(t^{2}),\\
w_{2}(\widehat z_{j}^{(2n)})+\sqrt{a_{n}b_{n}} & =
-\widehat z_j\left(1-\frac{\widetilde z_j^{(2n)}}{2} t+\bigO(t^{2})\right),\\
w_{1}(\widehat z_{j}^{(2n)})-\sqrt{a_{n}b_{n}} & =
-\widehat z_j\left(1+\frac{\widetilde z_j^{(2n)}}{2} t+\bigO(t^{2})\right).
\end{align*}
Plugging these estimates into (\ref{expl-c}) we get
$$2g_n(0)+2\ell_n=-2+\log(-4)+\bigO(t^{2}),$$
which together with (\ref{estim-err}) shows (\ref{limerr}) and the assertion about the constant $c_{n}$.
\hspace*{\fill}$\Box$
\section{Numerical experiments with interpolation points of modulus as large as the degree of the
interpolant}
\begin{figure}[htb!]
\centering
\includegraphics[width=6.5cm]{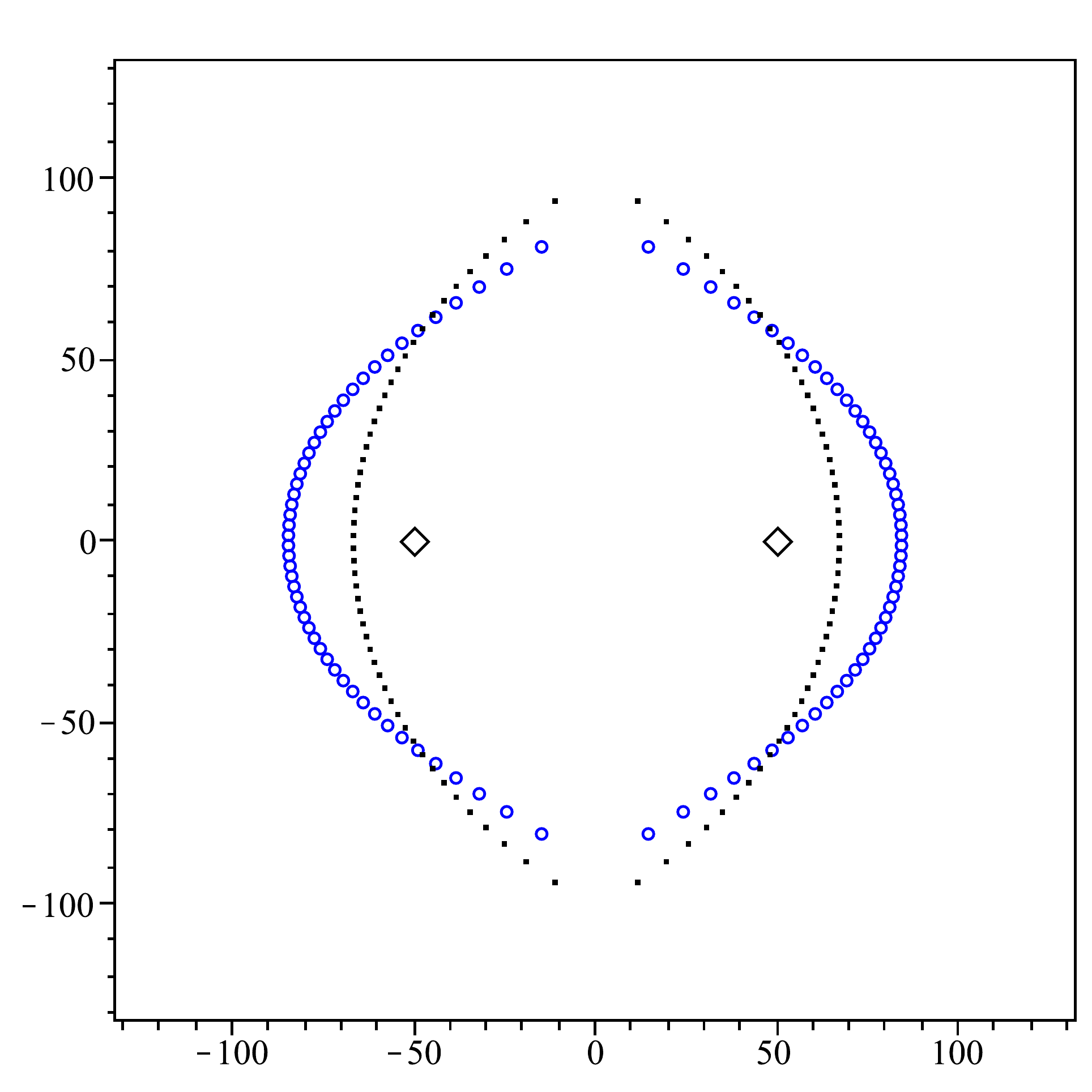}
\includegraphics[width=6.5cm]{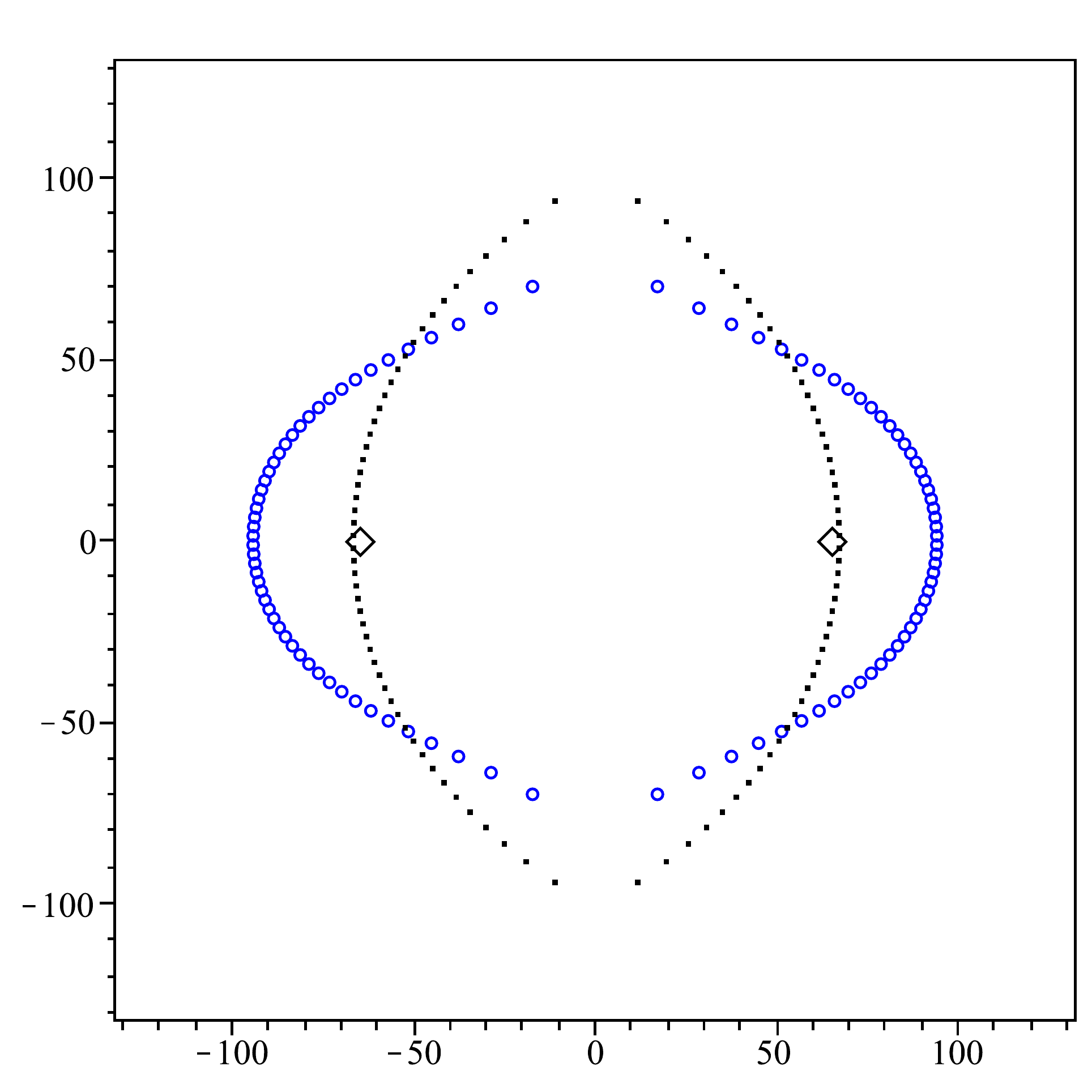}\\
\includegraphics[width=6.5cm]{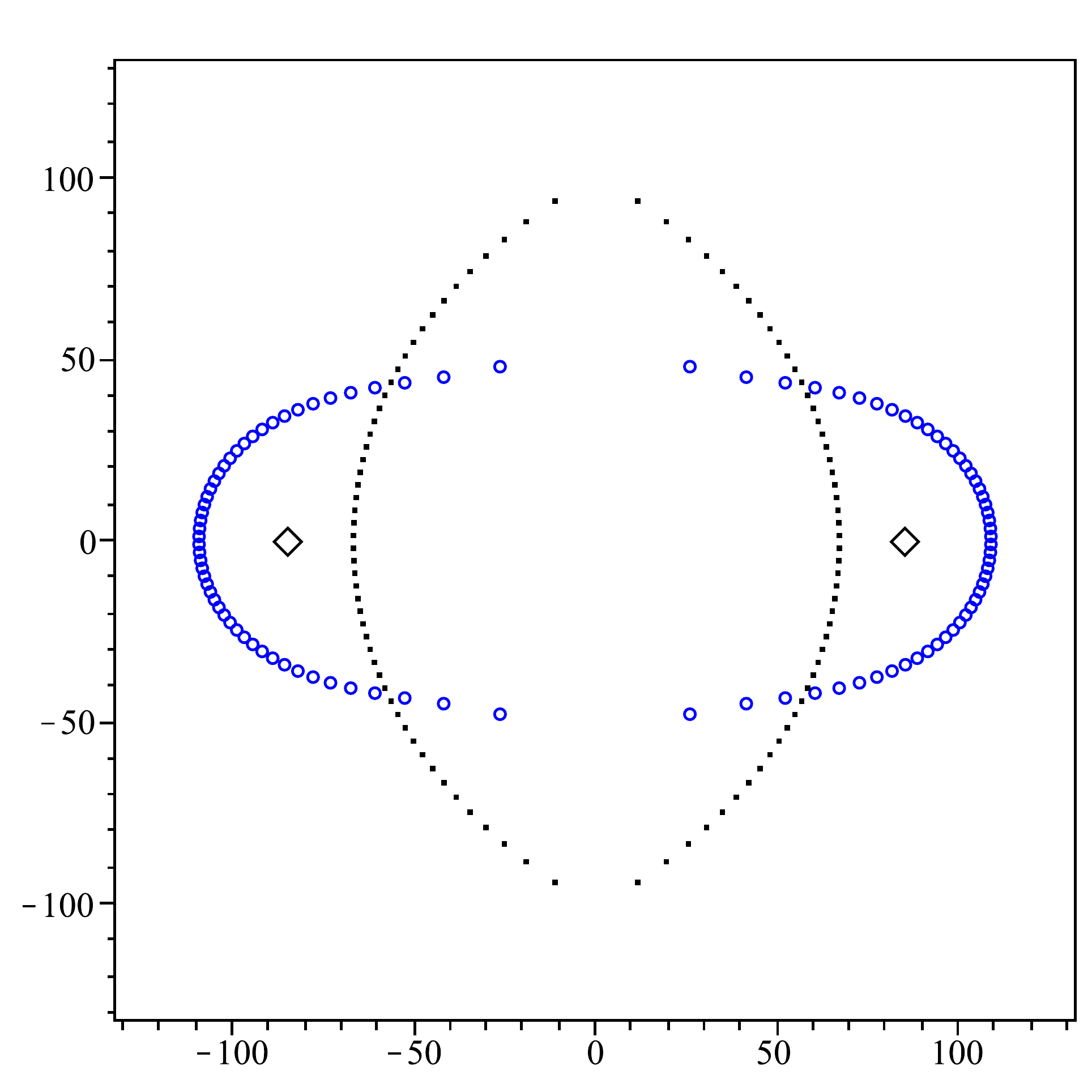}
\includegraphics[width=6.5cm]{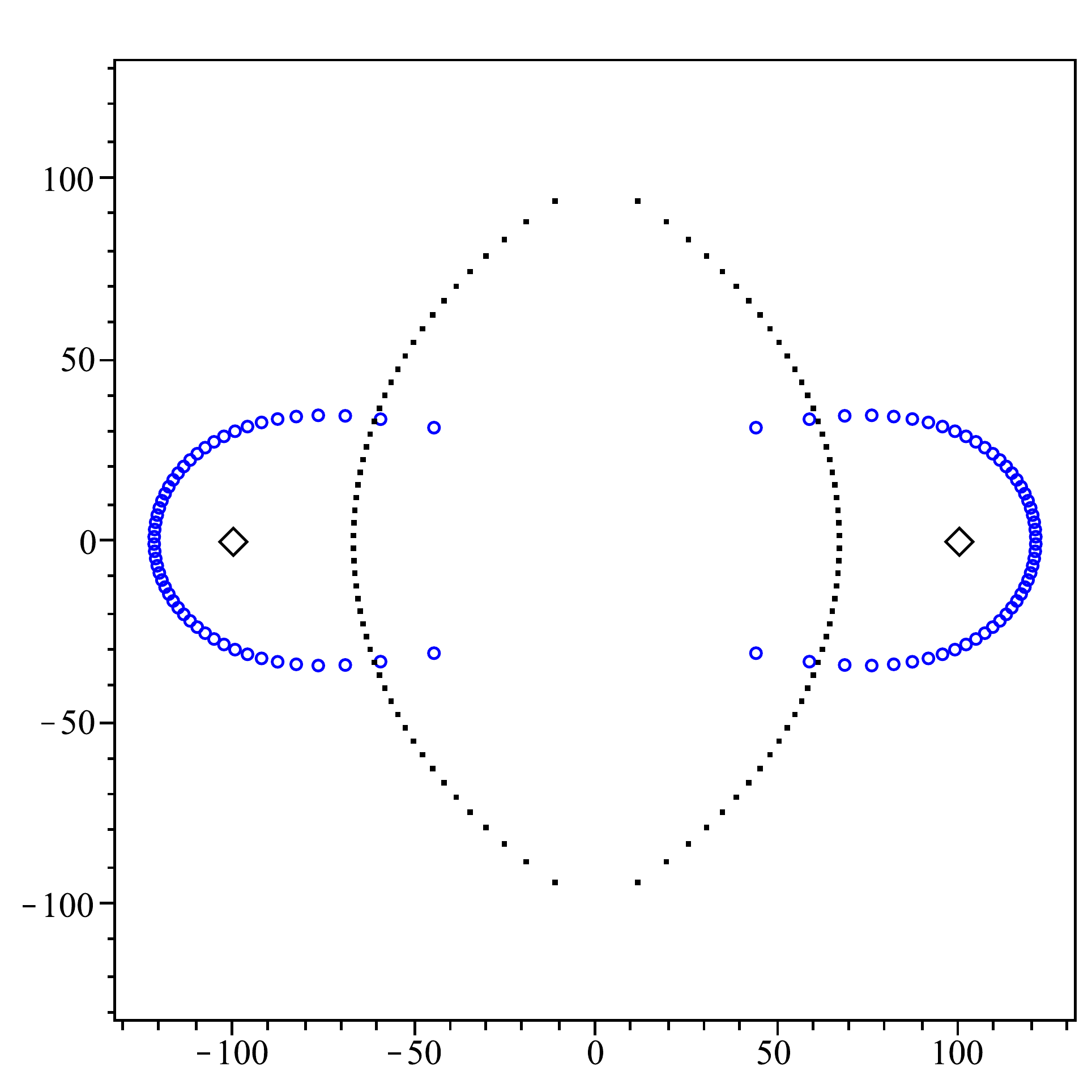}
\caption{Zeros and poles (circles in the left and right-half planes
respectively) of 2-point Pad\'e approximants of degree (51,50) with
two real interpolation points (diamonds) of multiplicity 51 located
at $\{-50,50\}$, $\{-65,65\}$, $\{-85,85\}$, $\{-100,100\}$. For
comparison, the zeros and poles of the Pad\'e approximant of degree
50 are shown with dots.} \label{2point}
\end{figure}

\begin{figure}[htb!]
\centering
\includegraphics[width=6.5cm]{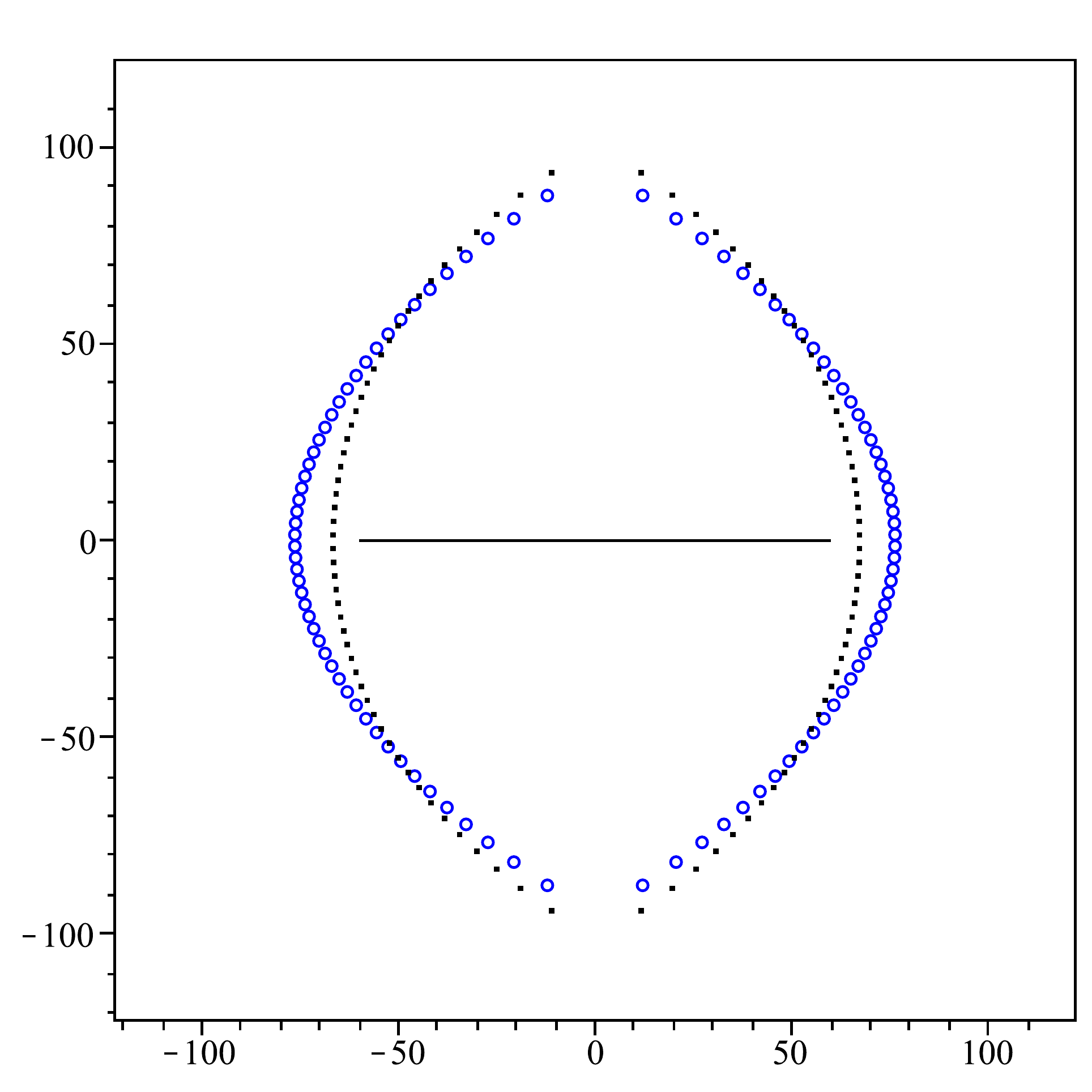}
\includegraphics[width=6.5cm]{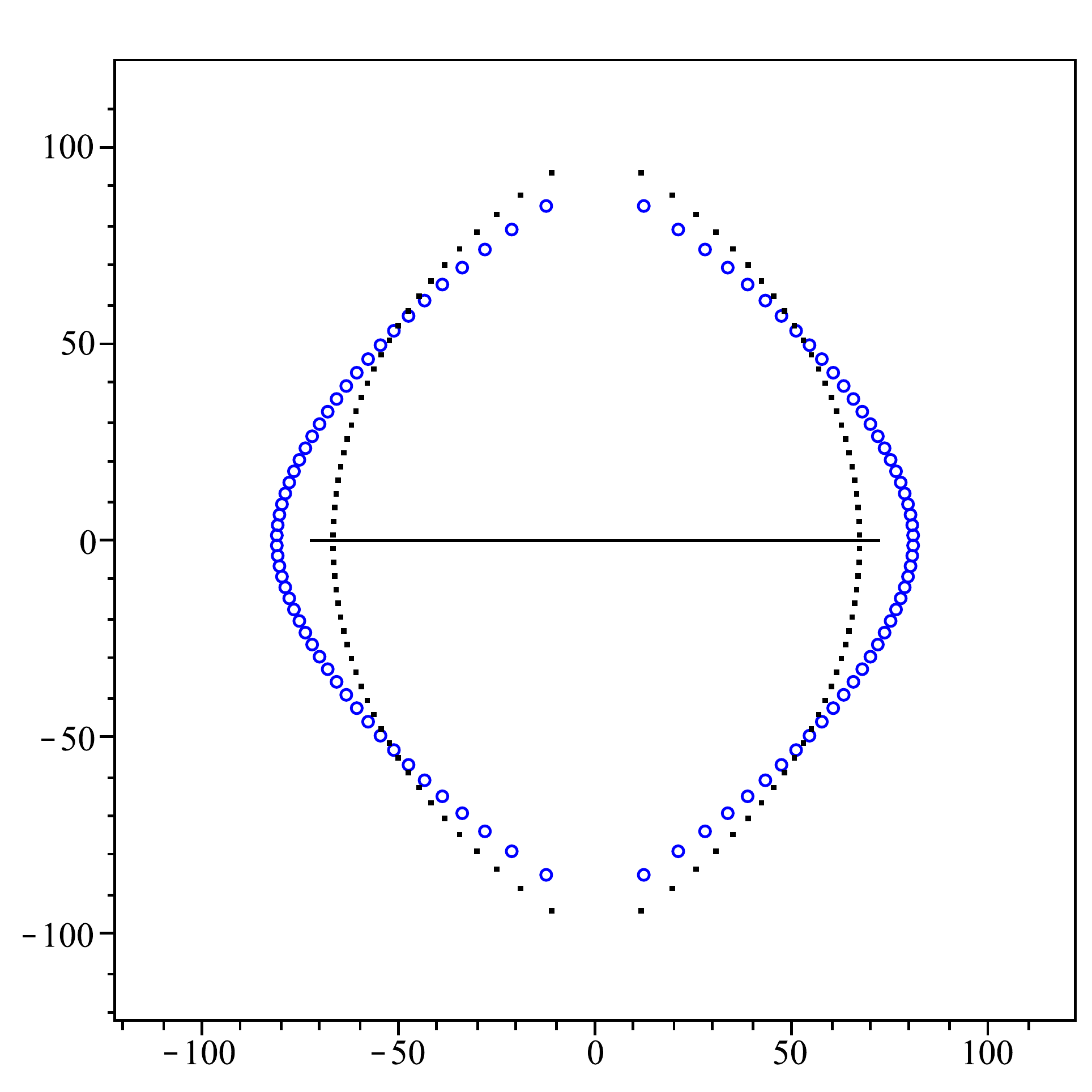}\\
\includegraphics[width=6.5cm]{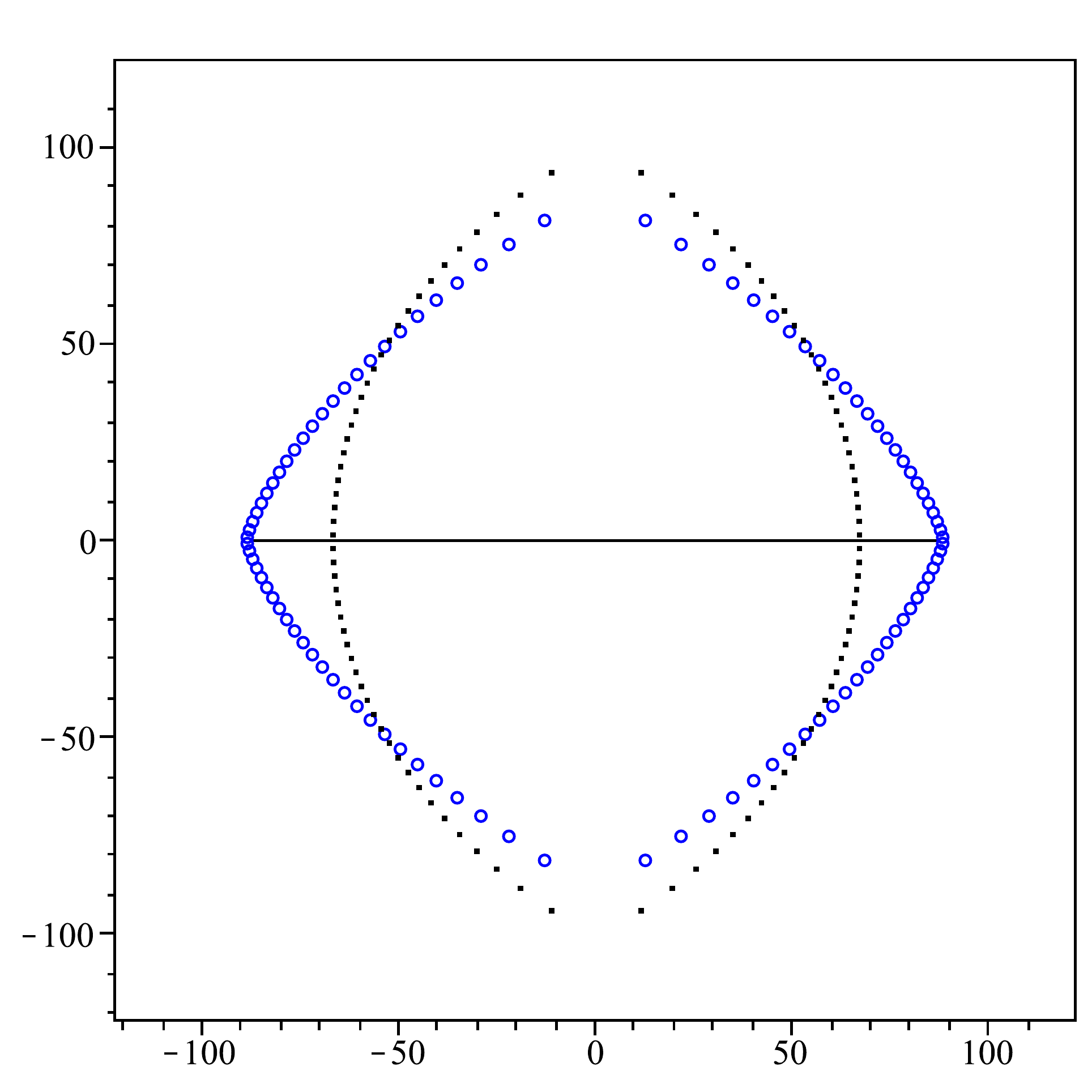}
\includegraphics[width=6.5cm]{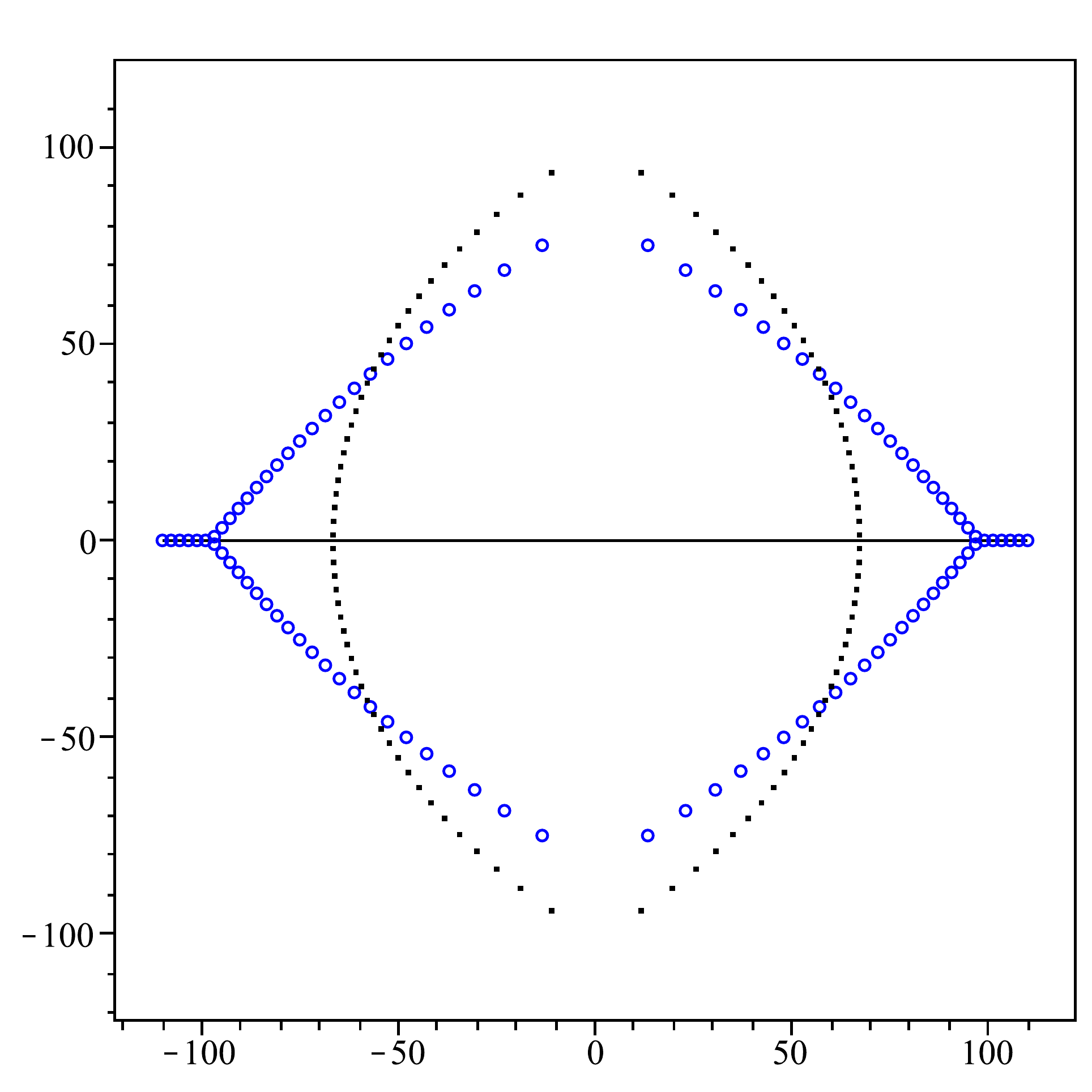}
\caption{Zeros and poles (circles in the left and right-half planes
respectively) of rational interpolants of degree 50 corresponding to
101 interpolation points regularly distributed on the real segments
$60I, 72.5I, 87.5I, 110I$, $I=[-1,1]$. For comparison, the zeros and
poles of the Pad\'e approximant of degree 50 are shown with dots.}
\label{line}
\end{figure}
In this section, we are interested in the location of zeros and
poles of rational interpolants associated to interpolation points
whose moduli are comparable to the degree of the interpolant. For
the particular case of shifted Pad\'e approximants
$p_{n}^{(\xi_{n})}/q_{n}^{(\xi_{n})}$ interpolating the exponential
function at $\xi_{n}$, we have the simple relation
\begin{equation*}
\frac{p_{n}^{(\xi_{n})}(z)}{q_{n}^{(\xi_{n})}(z)}=e^{\xi_{n}}\frac{p_{n}^{(0)}(z-\xi_{n})}
{q_{n}^{(0)}(z-\xi_{n})},
\end{equation*}
where $p_{n}^{(0)}(z)/q_{n}^{(0)}(z)$ denotes the usual Pad\'e
approximant interpolating $e^{z}$ at 0. Hence, we know at once that
the distributions of poles and zeros follow exactly the shift of the
interpolation point $\xi_{n}$ . In particular, the limit
distributions of the scaled (by $2n$) zeros and poles are modified
if and only if $\xi_{n}/n$ does not tend to 0, as $n$ tends to
infinity.

In Figure \ref{2point}, we consider the case of 2-point Pad\'e
approximants with two real symmetric interpolation points of equal
multiplicities. As the points approach the zeros and poles of the
usual Pad\'e approximant (or equivalently, in the scaled situation,
the critical curves $\gamma_{1}$ and $\gamma_{2}$), we can see how
the distributions of zeros and poles are modified. Clearly, a
repulsion takes place between the interpolation points and
the zeros and poles of the approximants. In Figure \ref{line}, we
consider the case of interpolation points regularly distributed on a
real segment. As the segment approaches the zeros and poles of the
usual Pad\'e approximant, we can again observe how the distributions
of zeros and poles are modified. Finally, in Figure \ref{circle},
the case of interpolation points regularly distributed on a circle
is depicted. The zeros and poles of the interpolants seem to be
pushed away as the circle intersects the limit distributions
corresponding to the usual Pad\'e approximant.
\begin{figure}[htb!]
\centering
\includegraphics[width=6.5cm]{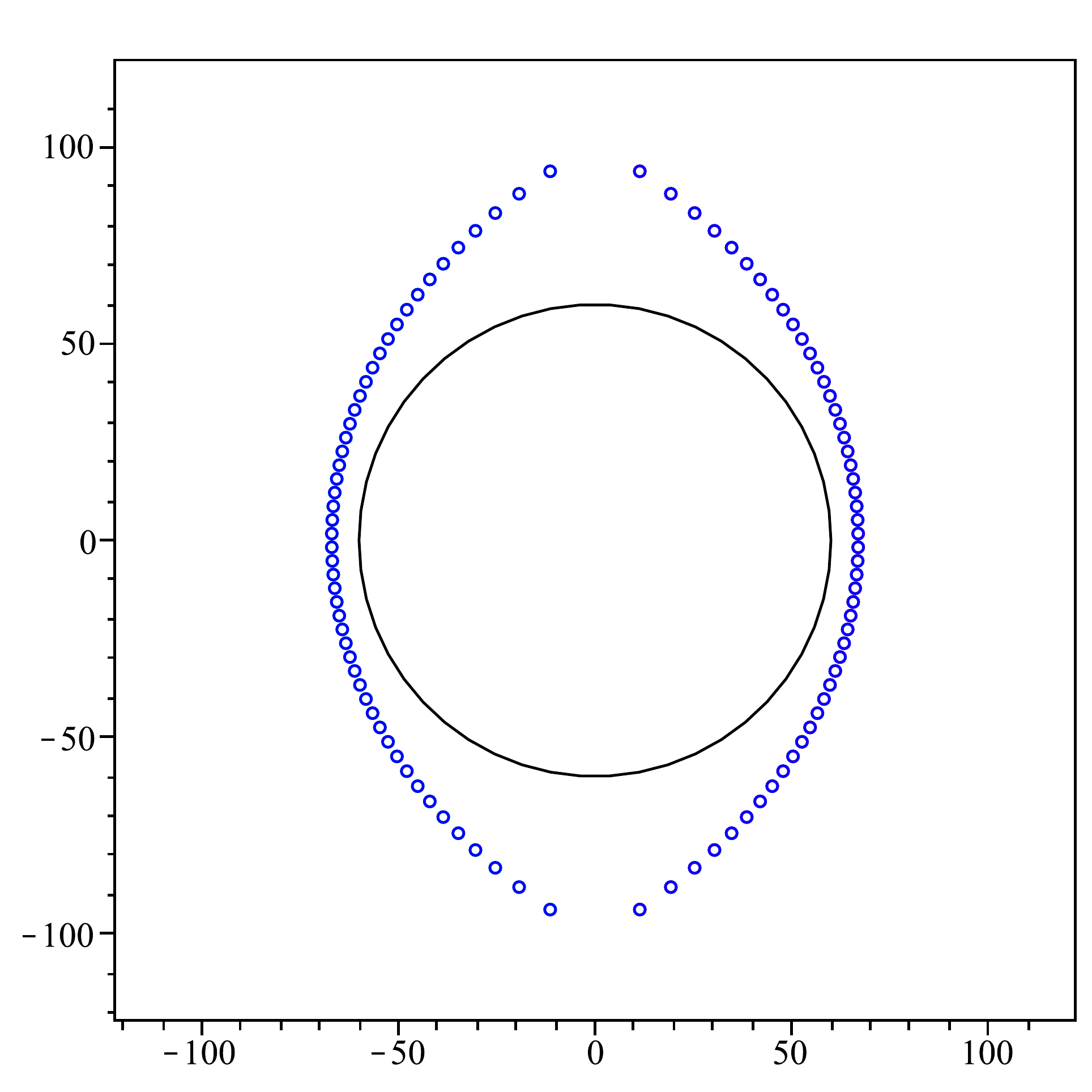}
\includegraphics[width=6.5cm]{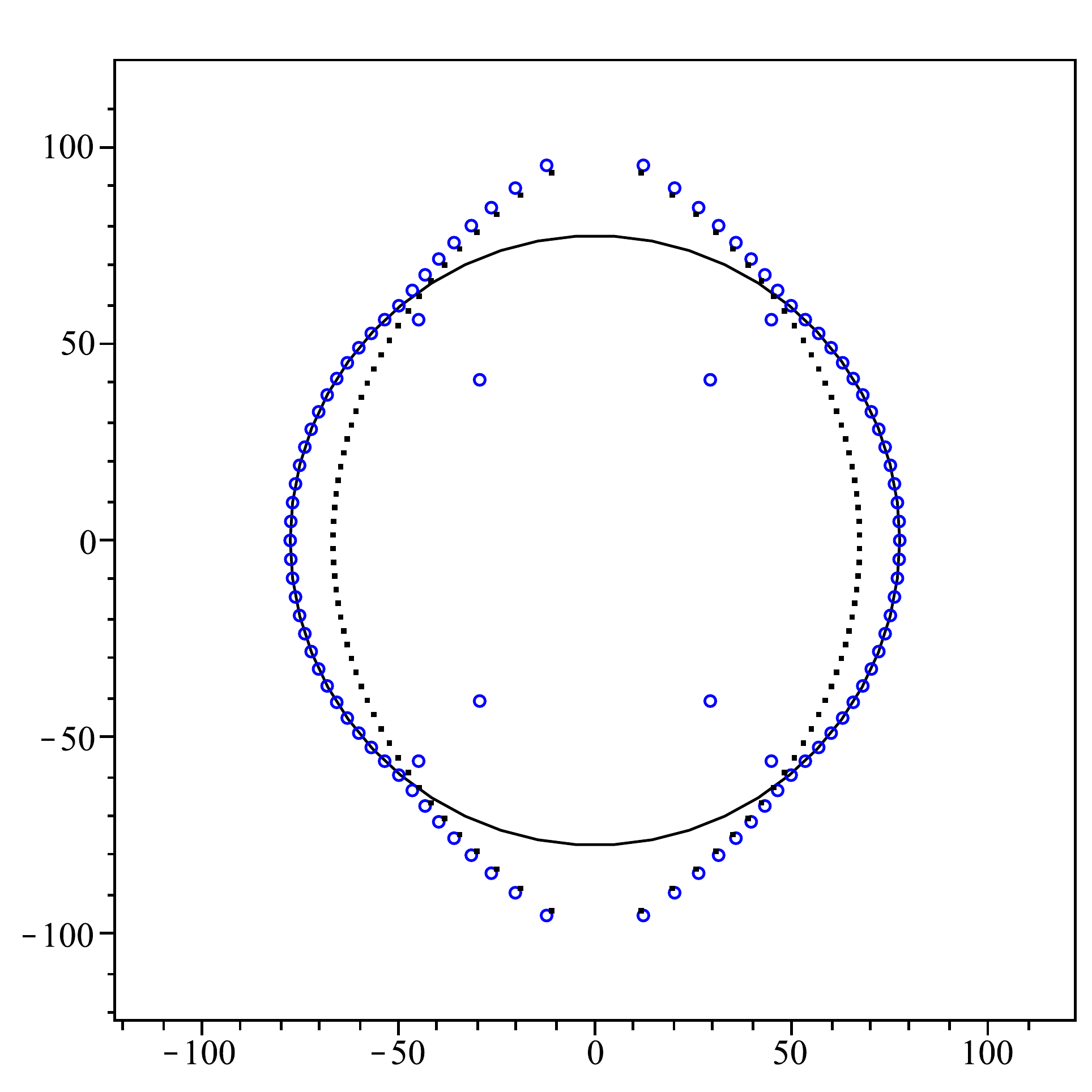}\\
\includegraphics[width=6.5cm]{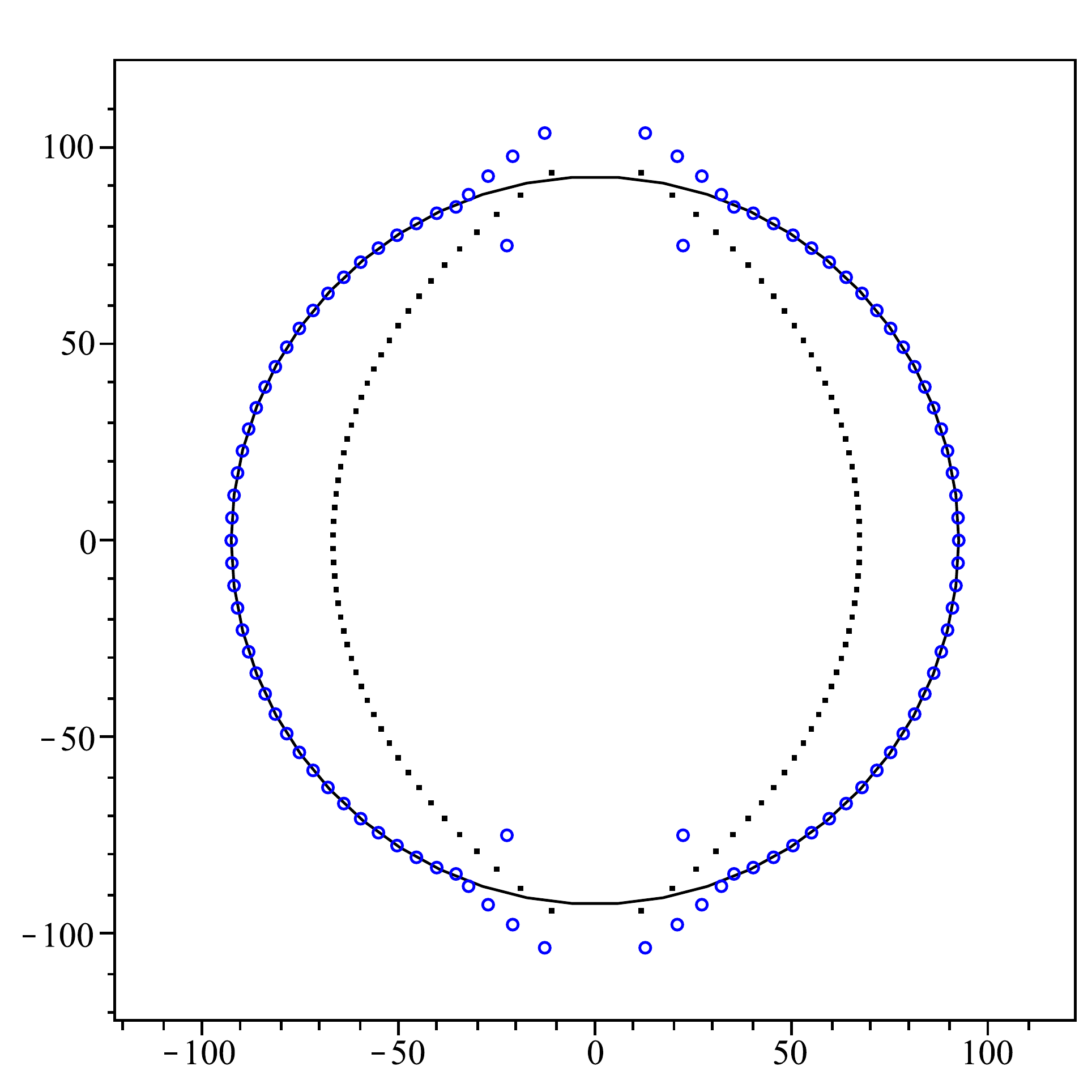}
\includegraphics[width=6.5cm]{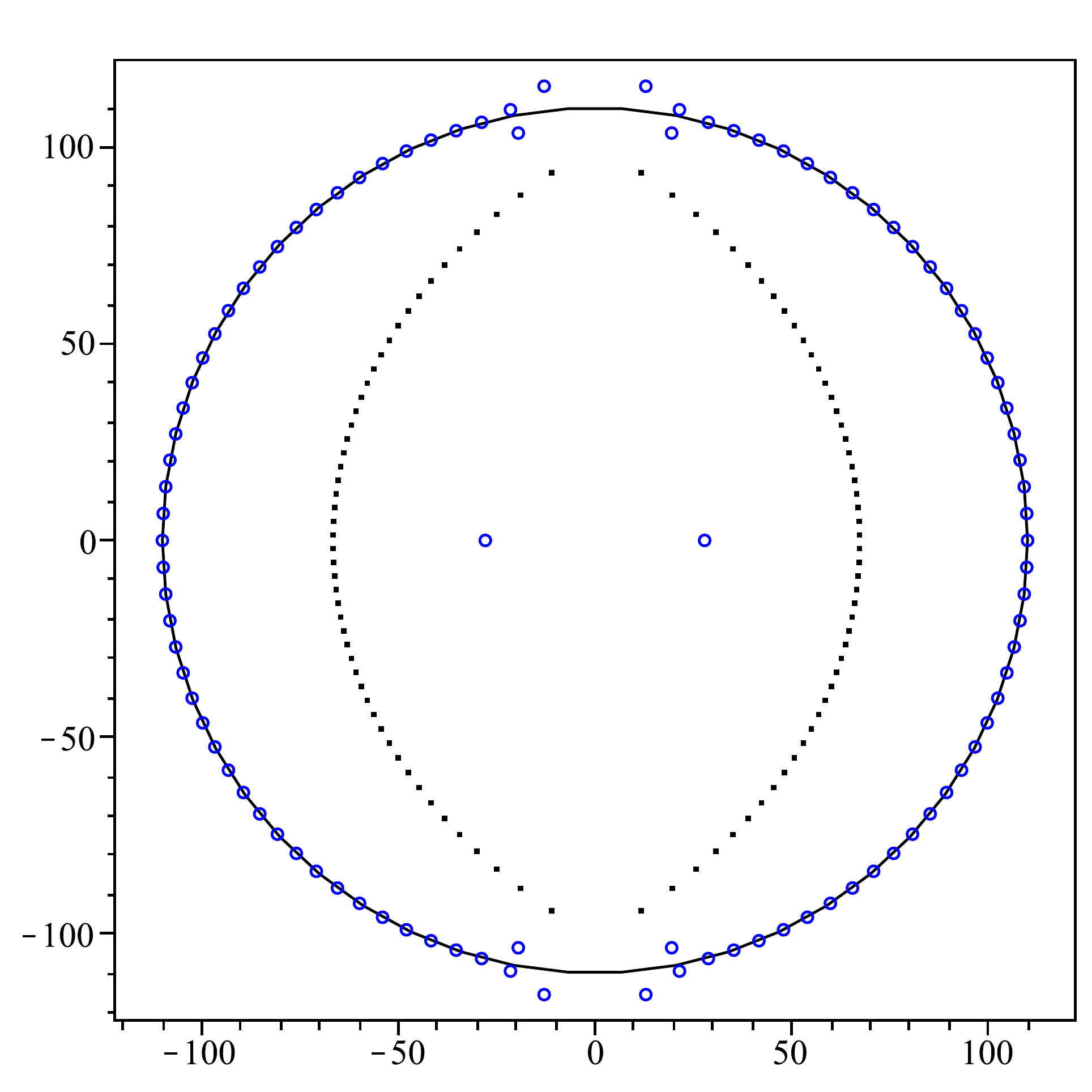}
\caption{Zeros and poles (circles in the left and right-half planes
respectively) of rational interpolants of degree 50 corresponding to
101 interpolation points regularly distributed on the circles of
radius 60, 77.5, 92.5, and 110. For comparison, the zeros and poles
of the Pad\'e approximant of degree 50 are shown with dots.}
\label{circle}
\end{figure}

From a theoretical point of view, polynomials whose roots are the
above zeros and poles  still satisfy the orthogonality relations
(\ref{orthog}). We note that, in the corresponding potential
(\ref{Vn}), the sum of log terms becomes preponderant as the
interpolation points grow faster with $n$. This should account for
the modification in the zeros and poles distributions of the
rational interpolants, that we observe in our experiments. In any
case, it would be interesting to study in more detail the
interaction between the interpolation points and the zeros and poles
of the approximants.

\section*{Acknowledgements}
TC acknowledges support by the Belgian Interuniversity
Attraction Pole P06/02 and by the ERC program FroM-PDE.

\obeylines
\texttt{Tom Claeys, tom.claeys@uclouvain.be
Universit\'e Catholique de Louvain
Chemin du cyclotron 2
B-1348 Louvain-La-Neuve, BELGIUM
\medskip
Franck Wielonsky, wielonsky@cmi.univ-mrs.fr
Laboratoire LATP - UMR CNRS 6632 Universit\'e de Provence
CMI 39 Rue Joliot Curie
F-13453 Marseille Cedex 20, FRANCE }

\end{document}

%% file: traject.pdf_tex

\begingroup
  \makeatletter
  \providecommand\color[2][]{%
    \errmessage{(Inkscape) Color is used for the text in Inkscape, but the package 'color.sty' is not loaded}
    \renewcommand\color[2][]{}%
  }
  \providecommand\transparent[1]{%
    \errmessage{(Inkscape) Transparency is used (non-zero) for the text in Inkscape, but the package 'transparent.sty' is not loaded}
    \renewcommand\transparent[1]{}%
  }
  \providecommand\rotatebox[2]{#2}
  \ifx\svgwidth\undefined
    \setlength{\unitlength}{560pt}
  \else
    \setlength{\unitlength}{\svgwidth}
  \fi
  \global\let\svgwidth\undefined
  \makeatother
  \begin{picture}(1,0.74895892)%
    \put(0,0){\includegraphics[width=\unitlength]{traject.pdf}}%
    \put(0.15887405,0.50563794){\color[rgb]{0,0,0}\makebox(0,0)[lb]{\smash{0.5}}}%
    \put(0.18034352,0.63588603){\color[rgb]{0,0,0}\makebox(0,0)[lb]{\smash{1}}}%
    \put(0.15028625,0.24943568){\color[rgb]{0,0,0}\makebox(0,0)[lb]{\smash{-0.5}}}%
    \put(0.17318702,0.12777538){\color[rgb]{0,0,0}\makebox(0,0)[lb]{\smash{-1}}}%
    \put(0.50954197,0.0476227){\color[rgb]{0,0,0}\makebox(0,0)[lb]{\smash{0}}}%
    \put(0.62547711,0.04476009){\color[rgb]{0,0,0}\makebox(0,0)[lb]{\smash{0.5}}}%
    \put(0.76717557,0.04762269){\color[rgb]{0,0,0}\makebox(0,0)[lb]{\smash{1}}}%
    \put(0.36927479,0.04762269){\color[rgb]{0,0,0}\makebox(0,0)[lb]{\smash{-0.5}}}%
    \put(0.25477099,0.0461914){\color[rgb]{0,0,0}\makebox(0,0)[lb]{\smash{-1}}}%
    \put(0.18034352,0.37682117){\color[rgb]{0,0,0}\makebox(0,0)[lb]{\smash{0}}}%
    \put(0.48950383,0.39113414){\color[rgb]{0,0,0}\makebox(0,0)[lb]{\smash{$D_0$}}}%
    \put(0.28769084,0.61298528){\color[rgb]{0,0,0}\makebox(0,0)[lb]{\smash{$D_{1,\infty}$}}}%
    \put(0.64837786,0.1750082){\color[rgb]{0,0,0}\makebox(0,0)[lb]{\smash{$D_{2,\infty}$}}}%
    \put(0.54532445,0.65019903){\color[rgb]{0,0,0}\makebox(0,0)[lb]{\smash{$i$}}}%
    \put(0.54103054,0.12205017){\color[rgb]{0,0,0}\makebox(0,0)[lb]{\smash{$-i$}}}%
    \put(0.28482821,0.40830972){\color[rgb]{0,0,0}\makebox(0,0)[lb]{\smash{$\gamma_1$}}}%
    \put(0.70133589,0.41260361){\color[rgb]{0,0,0}\makebox(0,0)[lb]{\smash{$\gamma_2$}}}%
  \end{picture}%
\endgroup

%% file: contour.pdf_tex

\begingroup
  \makeatletter
  \providecommand\color[2][]{%
    \errmessage{(Inkscape) Color is used for the text in Inkscape, but the package 'color.sty' is not loaded}
    \renewcommand\color[2][]{}%
  }
  \providecommand\transparent[1]{%
    \errmessage{(Inkscape) Transparency is used (non-zero) for the text in Inkscape, but the package 'transparent.sty' is not loaded}
    \renewcommand\transparent[1]{}%
  }
  \providecommand\rotatebox[2]{#2}
  \ifx\svgwidth\undefined
    \setlength{\unitlength}{274.34814729pt}
  \else
    \setlength{\unitlength}{\svgwidth}
  \fi
  \global\let\svgwidth\undefined
  \makeatother
  \begin{picture}(1,0.7006371)%
    \put(0,0){\includegraphics[width=\unitlength]{contour.pdf}}%
    \put(0.36110673,0.41264206){\color[rgb]{0,0,0}\makebox(0,0)[lb]{\smash{$b_n$}}}%
    \put(0.54212183,0.24325693){\color[rgb]{0,0,0}\makebox(0,0)[lb]{\smash{$\widetilde{\gamma}_{2,n}$}}}%
    \put(0.41091593,0.64554342){\color[rgb]{0,0,0}\makebox(0,0)[lb]{\smash{$\gamma''_{1,n}$}}}%
    \put(0.20808145,0.63310981){\color[rgb]{0,0,0}\makebox(0,0)[lb]{\smash{$\gamma_{1,n}$}}}%
    \put(0.05760651,0.49379536){\color[rgb]{0,0,0}\makebox(0,0)[lb]{\smash{$\gamma'_{1,n}$}}}%
    \put(0.62563966,0.62108986){\color[rgb]{0,0,0}\makebox(0,0)[lb]{\smash{$\partial U^{(-)}$}}}%
    \put(0.26684147,0.33125534){\color[rgb]{0,0,0}\makebox(0,0)[lb]{\smash{$-i$}}}%
  \end{picture}%
\endgroup

%% file: jumps.pdf_tex

\begingroup
  \makeatletter
  \providecommand\color[2][]{%
    \errmessage{(Inkscape) Color is used for the text in Inkscape, but the package 'color.sty' is not loaded}
    \renewcommand\color[2][]{}%
  }
  \providecommand\transparent[1]{%
    \errmessage{(Inkscape) Transparency is used (non-zero) for the text in Inkscape, but the package 'transparent.sty' is not loaded}
    \renewcommand\transparent[1]{}%
  }
  \providecommand\rotatebox[2]{#2}
  \ifx\svgwidth\undefined
    \setlength{\unitlength}{312.02592409pt}
  \else
    \setlength{\unitlength}{\svgwidth}
  \fi
  \global\let\svgwidth\undefined
  \makeatother
  \begin{picture}(1,0.54032509)%
    \put(0,0){\includegraphics[width=\unitlength]{jumps.pdf}}%
    \put(0.75600607,0.33068457){\color[rgb]{0,0,0}\makebox(0,0)[lb]{\smash{$\begin{pmatrix}1 & 1
\\0 & 1\end{pmatrix}$}}}%
    \put(0.40355887,0.52443429){\color[rgb]{0,0,0}\makebox(0,0)[lb]{\smash{$\begin{pmatrix}1 & 0\\1 & 1\end{pmatrix}$}}}%
    \put(-0.00175535,0.32524776){\color[rgb]{0,0,0}\makebox(0,0)[lb]{\smash{$\begin{pmatrix}0 & 1\\-1 & 0\end{pmatrix}$}}}%
    \put(0.40355887,0.01897871){\color[rgb]{0,0,0}\makebox(0,0)[lb]{\smash{$\begin{pmatrix}1 & 0\\1 & 1\end{pmatrix}$}}}%
    \put(0.52718215,0.21651182){\color[rgb]{0,0,0}\makebox(0,0)[lb]{\smash{0}}}%
    \put(0.52102688,0.30744293){\color[rgb]{0,0,0}\makebox(0,0)[lb]{\smash{$2\pi/3$}}}%
    \put(0.68231875,0.46782864){\color[rgb]{0,0,0}\makebox(0,0)[lb]{\smash{$\Gamma_1$}}}%
    \put(0.16219791,0.46873476){\color[rgb]{0,0,0}\makebox(0,0)[lb]{\smash{$\Gamma_2$}}}%
    \put(0.16582239,0.13346525){\color[rgb]{0,0,0}\makebox(0,0)[lb]{\smash{$\Gamma_3$}}}%
    \put(0.68866166,0.13074683){\color[rgb]{0,0,0}\makebox(0,0)[lb]{\smash{$\Gamma_4$}}}%
  \end{picture}%
\endgroup

%% file: maps.pdf_tex

\begingroup
  \makeatletter
  \providecommand\color[2][]{%
    \errmessage{(Inkscape) Color is used for the text in Inkscape, but the package 'color.sty' is not loaded}
    \renewcommand\color[2][]{}%
  }
  \providecommand\transparent[1]{%
    \errmessage{(Inkscape) Transparency is used (non-zero) for the text in Inkscape, but the package 'transparent.sty' is not loaded}
    \renewcommand\transparent[1]{}%
  }
  \providecommand\rotatebox[2]{#2}
  \ifx\svgwidth\undefined
    \setlength{\unitlength}{470.92224797pt}
  \else
    \setlength{\unitlength}{\svgwidth}
  \fi
  \global\let\svgwidth\undefined
  \makeatother
  \begin{picture}(1,0.35129426)%
    \put(0,0){\includegraphics[width=\unitlength]{maps.pdf}}%
    \put(0.14576927,0.21874814){\color[rgb]{0,0,0}\makebox(0,0)[lb]{\smash{$b_{n}$}}}%
    \put(0.22871456,0.12232031){\color[rgb]{0,0,0}\makebox(0,0)[lb]{\smash{$\widetilde{\gamma}_{2,n}$}}}%
    \put(0.16587407,0.32594505){\color[rgb]{0,0,0}\makebox(0,0)[lb]{\smash{$\gamma''_{1,n}$}}}%
    \put(0.08400268,0.32022224){\color[rgb]{0,0,0}\makebox(0,0)[lb]{\smash{$\gamma_{1,n}$}}}%
    \put(0.01532838,0.22290873){\color[rgb]{0,0,0}\makebox(0,0)[lb]{\smash{$\gamma'_{1,n}$}}}%
    \put(0.25254446,0.31468985){\color[rgb]{0,0,0}\makebox(0,0)[lb]{\smash{$\partial U^{(-)}$}}}%
    \put(0.10772038,0.18128845){\color[rgb]{0,0,0}\makebox(0,0)[lb]{\smash{$-i$}}}%
    \put(0.81598605,0.12949701){\color[rgb]{0,0,0}\makebox(0,0)[lb]{\smash{$\widetilde{\lambda}_{2,n}$}}}%
    \put(0.66589556,0.30928425){\color[rgb]{0,0,0}\makebox(0,0)[lb]{\smash{$\lambda''_{1,n}$}}}%
    \put(0.57281051,0.2081848){\color[rgb]{0,0,0}\makebox(0,0)[lb]{\smash{$\lambda_{1,n}$}}}%
    \put(0.63555396,0.07852077){\color[rgb]{0,0,0}\makebox(0,0)[lb]{\smash{$\lambda'_{1,n}$}}}%
    \put(0.72372558,0.14764504){\color[rgb]{0,0,0}\makebox(0,0)[lb]{\smash{$0$}}}%
    \put(0.35838216,0.00480772){\color[rgb]{0,0,0}\makebox(0,0)[lb]{\smash{$f_n(z)+s_n(z)$}}}%
    \put(0.74008507,0.25663069){\color[rgb]{0,0,0}\makebox(0,0)[lb]{\smash{$S_{1,n}$}}}%
    \put(0.59938179,0.24147802){\color[rgb]{0,0,0}\makebox(0,0)[lb]{\smash{$S_{2,n}$}}}%
    \put(0.59000156,0.11809197){\color[rgb]{0,0,0}\makebox(0,0)[lb]{\smash{$S_{3,n}$}}}%
    \put(0.76173184,0.08706507){\color[rgb]{0,0,0}\makebox(0,0)[lb]{\smash{$S_{4,n}$}}}%
    \put(0.20613389,0.24652891){\color[rgb]{0,0,0}\makebox(0,0)[lb]{\smash{$R_{1,n}$}}}%
    \put(0.07914008,0.12025665){\color[rgb]{0,0,0}\makebox(0,0)[lb]{\smash{$R_{4,n}$}}}%
    \put(0.10872387,0.2970378){\color[rgb]{0,0,0}\makebox(0,0)[lb]{\smash{$R_{2,n}$}}}%
    \put(0.03801138,0.26745401){\color[rgb]{0,0,0}\makebox(0,0)[lb]{\smash{$R_{3,n}$}}}%
  \end{picture}%
\endgroup

%% file: contourR.pdf_tex

\begingroup
  \makeatletter
  \providecommand\color[2][]{%
    \errmessage{(Inkscape) Color is used for the text in Inkscape, but the package 'color.sty' is not loaded}
    \renewcommand\color[2][]{}%
  }
  \providecommand\transparent[1]{%
    \errmessage{(Inkscape) Transparency is used (non-zero) for the text in Inkscape, but the package 'transparent.sty' is not loaded}
    \renewcommand\transparent[1]{}%
  }
  \providecommand\rotatebox[2]{#2}
  \ifx\svgwidth\undefined
    \setlength{\unitlength}{321.86118164pt}
  \else
    \setlength{\unitlength}{\svgwidth}
  \fi
  \global\let\svgwidth\undefined
  \makeatother
  \begin{picture}(1,0.75247893)%
    \put(0,0){\includegraphics[width=\unitlength]{contourR.pdf}}%
    \put(0.77419419,0.66990672){\color[rgb]{0,0,0}\makebox(0,0)[lb]{\smash{$\widetilde\gamma_2$}}}%
    \put(0.20606983,0.48881714){\color[rgb]{0,0,0}\makebox(0,0)[lb]{\smash{$\gamma_1''$}}}%
    \put(-0.00165056,0.48349099){\color[rgb]{0,0,0}\makebox(0,0)[lb]{\smash{$\gamma_1'$}}}%
    \put(0.22914993,0.73737148){\color[rgb]{0,0,0}\makebox(0,0)[lb]{\smash{$\partial U^{(+)}$}}}%
    \put(0.26110686,0.0041361){\color[rgb]{0,0,0}\makebox(0,0)[lb]{\smash{$\partial U^{(-)}$}}}%
    \put(0.28063621,0.64860208){\color[rgb]{0,0,0}\makebox(0,0)[lb]{\smash{$\bullet$}}}%
    \put(0.29223926,0.10649941){\color[rgb]{0,0,0}\makebox(0,0)[lb]{\smash{$\bullet$}}}%
    \put(0.30726698,0.66280515){\color[rgb]{0,0,0}\makebox(0,0)[lb]{\smash{$i$}}}%
    \put(0.27708539,0.07337633){\color[rgb]{0,0,0}\makebox(0,0)[lb]{\smash{$-i$}}}%
  \end{picture}%
\endgroup